\documentclass[a4paper,11pt,leqno]{article}

\usepackage[latin1]{inputenc}
\usepackage[T1]{fontenc} 
\usepackage[english]{babel}
\usepackage[notcite,notref]{showkeys}
\usepackage{verbatim}
\usepackage{amsmath}
\usepackage{amsthm}
\usepackage{amssymb}
\usepackage{bm}
\usepackage{amsfonts}
\usepackage{mathrsfs}
\usepackage{latexsym}

\theoremstyle{definition}
\newtheorem{defin}{Definition}[section]

\newtheorem{rem}[defin]{Remark}

\theoremstyle{plane}
\newtheorem{theo}[defin]{Theorem}
\newtheorem{prop}[defin]{Proposition}
\newtheorem{coroll}[defin]{Corollary}
\newtheorem{lemma}[defin]{Lemma}

\newcommand{\mbb}{\mathbb}

\newcommand{\mc}{\mathcal}
\newcommand{\veps}{\varepsilon}

\newcommand{\what}{\widehat}
\newcommand{\wtilde}{\widetilde}
\newcommand{\vphi}{\varphi}
\newcommand{\oline}{\overline}
\newcommand{\ra}{\rightarrow}
\newcommand{\upra}{\rightharpoonup}
\newcommand{\hra}{\hookrightarrow}

\newcommand{\R}{\mathbb{R}}

\newcommand{\N}{\mathbb{N}}
\newcommand{\Z}{\mathbb{Z}}

\renewcommand{\div}{{\rm div}\,}
\renewcommand{\det}{{\rm det}}

\newcommand{\Id}{{\rm Id}\,}

\def\d{\partial}
\def\div{{\rm div}\,}

\title{\large{\bfseries{CONSERVATION OF GEOMETRIC STRUCTURES FOR NON-HOMOGENEOUS INVISCID INCOMPRESSIBLE FLUIDS}}}

\author{\textsl{Francesco Fanelli} \vspace{0.3cm}\\
\small{\textsc{SISSA}}\\
\footnotesize{Functional Analysis and Applications Sector}\\
\footnotesize{\textit{via Bonomea 265 - 34136 Trieste, ITALY}} \vspace{0.2cm}\\
\small{\textsc{Universit\'e Paris-Est}}\\
\footnotesize{Laboratoire d'Analyse et de Math\'ematiques Appliqu\'ees, UMR CNRS 8050}\\
\footnotesize{\textit{61 Avenue du G\'en\'eral De Gaulle - 94010 Cr\'eteil Cedex, FRANCE}} \vspace{0.2cm}\\
\small{\ttfamily{francesco.fanelli@sissa.it} \hspace{1cm} francesco.fanelli@math.cnrs.fr}}

\date\today

\begin{document}

\maketitle

\subsubsection*{Abstract}
{\small  We obtain a result about propagation of geometric properties for solutions
of the non-homogeneous incompressible Euler system in any dimension $N\geq2$. In particular, we investigate conservation
of striated and conormal regularity, which is a natural way of generalizing the $2$-D structure of vortex patches. The results
we get are only local in time, even in the dimension $N=2$; however, we provide an explicit lower bound for
the lifespan of the solution. In the case of physical
dimension $N=2$ or $3$, we investigate also propagation of H\"older regularity in the interior of a bounded domain.}

\section{Introduction}

In this paper we are interested in studying conservation of geometric properties for solutions of the density-dependent
incompressible Euler system
\begin{equation}\label{eq:ddeuler}
\left\{\begin{array}{l}
\d_t\rho\,+\,u\cdot\nabla \rho\,=\,0\\[1ex]
\rho\left(\d_tu\,+\,u\cdot\nabla u\right)\,+\,\nabla\Pi\,=\,0\\[1ex]
\div u\,=\,0\,,
\end{array}\right.
\end{equation}
which describes the evolution of a non-homogeneous inviscid fluid with no body force acting on it, an assumption we will
make throughout all this paper to simplify the presentation.
Here, $\rho(t,x)\in\R_+$ represents the density of the fluid, $u(t,x)\in\R^N$ its velocity field and $\Pi(t,x)\in\R$ its pressure.
The term $\nabla\Pi$ can be also seen as the Lagrangian multiplier associated to the divergence-free constraint over the velocity.

We will always suppose that the variable $x$ belongs to the whole space $\,\R^N$.

\medskip
The problem of preserving geometric structures came out already in the homogeneous case, for which $\rho\equiv1$ and system
\eqref{eq:ddeuler} becomes
$$
\left\{\begin{array}{l}
\d_tu\,+\,u\cdot\nabla u\,+\,\nabla\Pi\,=\,0\\[1ex]
\div u\,=\,0\,,
\end{array}\right.
\leqno(E)
$$
in studying $2$-dimensional vortex patches, that is to say the initial vorticity $\Omega_0$ is the characteristic function of a
bounded domain $D_0$. As we will explain below, in the case of higher dimension $N\geq3$ this notion is generalized by
the properties of striated and conormal regularity.

The vorticity of the fluid is defined as the skew-symmetric matrix
\begin{equation} \label{def:vort}
\Omega\,:=\,\nabla u\,-\,^t\nabla u
\end{equation}
and in the homogenous case it satisfies the equation
$$
\d_t\Omega\,+\,u\cdot\nabla\Omega\,+\,\Omega\cdot\nabla u\,+\,^t\nabla u\cdot\Omega\,=\,0\,.
$$
In dimension $N=2$ it can be identified with the scalar function $\omega\,=\,\partial_1u^2\,-\,\partial_2u^1$,
while for $N=3$ with the vector-field $\omega\,=\,\nabla\times u$. Let us recall also that in the bidimensional case
this quantity is transported by the velocity field:
$$
\d_t\omega\,+\,u\cdot\nabla\omega\,=\,0\,.
$$

The notion of vortex patches was introduced in \cite{Y} and gained new interest after the survey paper \cite{M} of Majda.
In the case $\,N=2\,$ Yudovich's theorem ensures the existence of a unique global
solution of the homogeneous Euler system, which preserves the geometric structure: the vorticity remains the characteristic
function of the evolution (by the flow associated to this solution) of the domain $D_0$. Vortex patches in bounded domains of $\R^2$
were also studied by Depauw (see \cite{Dep}), while Dutrifoy in \cite{Dut} focused on the case of domains in $\R^3$.
Moreover, in \cite{Ch1991} Chemin proved that, if the initial domain has boundary $\d D_0$ of class
$\mc{C}^{1,\veps}$ for some $\veps>0$, then this regularity is preserved during the evolution for small times; in \cite{Ch1993} he
also showed a global in time persistence issue. In \cite{D1997-sing} Danchin
considered instead the case in which initial data of the Euler system are vortex patches with singular boundary: he proved that
if $\d D_0$ is regular apart from a closed subset, then it remains regular for all times, apart from the closed subset transported
by the flow associated to the solution.

In the case $N\geq3$ one can't expect to have global results anymore, nor to preserve the initial vortex patch structure, because
of the presence of the stretching term in the vorticity equation. Nevertheless, it's possible to introduce the definition of
\emph{striated regularity}, which generalizes in a quite natural way the previous one of vortex patch: it means that the vorticity is
more regular along some fixed directions, given by a nondegenerate family of vector-fields (see definition \ref{d:stri} below).
This notion was introduced first by Bony in \cite{Bo1979} in studying hyperbolic equations, and then adapted by Alinhac (see \cite{A})
and Chemin (see \cite{Ch1988}) for nonlinear partial differential equations.

In \cite{G-SR}, Gamblin and Saint-Raymond proved that striated regularity is preserved during the
evolution in any dimension $N\geq3$, but, as already remarked, only locally in time (see also \cite{S}). They also obtained
global results if the initial data have other nice properties (e.g., if the initial velocity is axisymmetric).

As Euler system is, in a certain sense, a limit case of the Navier-Stokes system as the viscosity of the fluid goes to $0$,
it's interesting to study if there is also ``convergence'' of the geometric properties of the solutions. Recently Danchin
proved results on striated regularity for the solutions of the Navier-Stokes system
$$
\left\{\begin{array}{l}
\d_tu\,+\,u\cdot\nabla u\,-\,\nu\,\Delta u\,+\,\nabla\Pi\,=\,0 \\[1ex]
\div u\,=\,0\,,
\end{array}\right.
\leqno(N\!S_\nu)
$$
in \cite{D1997} for the $2$-dimensional case, in \cite{D1999} for the general one. Already in the former paper, he had to dismiss the
vortex patch structure ``stricto sensu'' due to the presence of the viscous term, which comes out also in the vorticity equation and has
a smoothing effect;
however, he still got global in time results. Moreover, in both his works he had to handle with spaces of
type $B^{1+\veps}_{p,\infty}$ (with $p\in\,]2,+\infty[$ and $\veps\in\,]2/p,1[\,$) due to technical reasons which come
out with a viscous fluid. Let us immediately clarify that these problems have been recently solved by Hmidi in \cite{Hm}
(see also \cite{B-C-D}). In the above mentioned works
Danchin proved also a priori estimates for solutions of $(N\!S_\nu)$ independent of the viscosity $\nu$, therefore
preservation of the geometric structures in passing from solutions of $(N\!S_\nu)$ to solutions of $(E)$ in the limit $\nu\ra0$.

In this paper we come back to the inviscid case and we study the non-homogeneous incompressible Euler system
\eqref{eq:ddeuler}. We will prove that preservation of geometric properties of initial data, such as striated and conormal
regularity, still holds in this setting, as in the classical (homogeneous) one.

Let us point out that considering geometric structures is not only a generalization of the classical problem of vortex patches.
It can be also seen as an improvement to the well-posedness result for \eqref{eq:ddeuler} in critical
Besov spaces $B^s_{\infty,r}$ (see paper \cite{D-F}). As a matter of fact, here we will make lower regularity hypothesis on
the initial data, which will be compensated by the additional geometric assumptions.

Note that in the $2$-dimensional
case the equation for the vorticity reads
$$
\d_t\omega\,+\,u\cdot\nabla\omega\,+\,\nabla\left(\frac{1}{\rho}\right)\wedge\nabla\Pi\,=\,0\,,
$$
so it's not better than in higher dimension due to the presence of the density term, which doesn't allow us to get conservation of
Lebesgue norms. This is also the reason why it's not clear if Yudovich's theorem still holds true for non-homogeneous fluids: having
$\omega_0\in L^q\cap L^\infty$, combined with suitable hypothesis on $\rho_0$, doesn't give rise to a local solution. \\
So, we will immediately focus on the general case $N\geq2$. We will assume the initial velocity $u_0$ and the initial
vorticity $\Omega_0$ to be in some Lebesgue spaces, in order to assure the pressure term to belong to $L^2$, a
requirement we could not bypass.
As a matter of fact, $\nabla\Pi$ satisfies an elliptic equation with low regularity coefficient,
$$
-\,\div\left(a\,\nabla\Pi\right)\,=\,\div F\,,
$$
and it can be solved independently of $a$ only in the energy space $L^2$.
Moreover, we will suppose $\Omega_0$ to have regularity properties of geometric type.
Obviously, we will require some natural but quite general hypothesis also on the initial density $\rho_0$ of the fluid: we suppose
$\rho_0$ to be bounded with its gradient and that it satisfies geometric assumptions analogous to those for $\Omega_0$.
Let us point out that proving the velocity field to be Lipschitz, which was the key part in the homogeneous
case, works as in this setting: it relies on Biot-Savart law and it requires no further hypothesis on the density term.
Let us also remark that no smallness conditions over the density are needed. Of course, we will get only local in
time results. Moreover, we will see that geometric structures propagate also to the velocity field and to the pressure term.

\medskip
Our paper is organized in the following way.

In the first part, we will recall basic facts about Euler system: some properties of the vorticity and how to
associate a flow to the velocity field. In this section we will also give the definition of the geometric properties
we are studying and we will state the main results we got about striated and conormal regularity.

In section \ref{s:tools}, we will explain the mathematical tools, from Fourier Analysis, we need to prove our claims:
we will present the Littlewood-Paley decomposition and some techniques coming from paradifferential calculus. In particular,
we will introduce the notion of paravector-field, as defined in \cite{D1999}: it will play a fundamental role
in our analysis, because it is, in a certain sense, the principal part of the derivation operator along a fixed vector-field.
Moreover, we will also quote some results about
transport equations in H\"older spaces and about elliptic equations in divergence form with low regularity coefficients.

In section \ref{s:striated}, we will finally be able to tackle the proof of our result about striated regularity. First of all, we will
state a priori estimates for suitable smooth solutions of the Euler system \eqref{eq:ddeuler}. Then from them we will get, in a quite
classical way, the existence of a solution with the required properties: we will construct a sequence of regular solutions of
system \eqref{eq:ddeuler} with approximated data, and, using a compactness argument, we will show the convergence of this sequence
to a ``real'' solution. Proving preservation of the geometric structure requires instead strong convergence in
rough spaces of type $\mc{C}^{-\alpha}$ (for some $\alpha>0$). The uniqueness of the solution will follow from a stability result
for our equations. In the following section, we will also give an estimate from below for the lifespan of the solution.

Finally, we will spend a few words about conormal regularity: proving its propagation from the previous result is standard and
can be done as in the homogenous setting. As a consequence, inspired by what done in Huang's paper \cite{Hu}, in the physical
case of space dimension $N=2$ or $3$ we can improve our result: we will also show
that, if the initial data are H\"older continuous in the interior of a suitably smooth bounded domain, the solution
conserves this property during the time evolution, i.e. it is still H\"older continuous
in the interior of the domain transported by the flow.

\section{Basic definitions and main results}

Let $\left(\rho,u,\nabla\Pi\right)$ be a solution of the density-dependent incompressible Euler system \eqref{eq:ddeuler}
over $[0,T]\times\R^N$ and let us denote the vorticity of the fluid by $\Omega$.
As in the homogeneous case, it will play a fundamental role throughout all this paper, so let us spend a few words about it.

From the definition \eqref{def:vort}, it is obvious that, for all $q\in[1,+\infty]$, if $\,\nabla u\in L^q$, then also $\,\Omega\in L^q$.
Conversely, if  $u$ is divergence-free, then for all $1\leq i\leq N$ we have $\,\Delta u^i\,=\,\sum_{j=1}^N\partial_j\Omega_{ij}$, and so, formally,
\begin{equation} \label{eq:BS-law}
 u^i\,=\,-\left(-\Delta\right)^{-1}\sum^N_{j=1}\d_j\Omega_{ij}\,.
\end{equation}
This is the Biot-Savart law, and it says that a divergence free vector-field $u$ is completely determined by its vorticity.
From \eqref{eq:BS-law} we immediately get
\begin{equation} \label{eq:BS_grad}
\nabla u^i\,=\,-\nabla\left(-\Delta\right)^{-1}\sum_{j=1}^N\partial_j\Omega_{ij}\,.
\end{equation}
Now, as the symbol of the operator $-\d_i\left(-\Delta\right)^{-1}\d_j$ is $\sigma(\xi)=\xi_i\xi_j/|\xi|^2$,
the classical Calderon-Zygmund theorem ensures that\footnote{This time the extreme values of $q$ are not included.}
for all $q\in\,]1,+\infty[\,$,  if $\,\Omega\in L^q$ then $\,\nabla u\in L^q$ and
\begin{equation}\label{est:CZ}
\|\nabla u\|_{L^q}\leq C\,\frac{q^2}{q-1}\,\|\Omega\|_{L^q}\,.
\end{equation}

In dimension $N=2$ the vorticity equation is simpler than in the general case due to the absence of the stretching term. Nevertheless,
as remarked above, the exterior product involving density and pressure terms makes it impossible to get conservation of Lebesgue norms,
which was the fundamental issue to get global existence for the classical system $(E)$.
So, we immediately focus on the case $N\geq2$ whatever, in which the vorticity equation reads
\begin{equation} \label{eq:vort}
 \d_t\Omega\,+\,u\cdot\nabla\Omega\,+\,\Omega\cdot\nabla u\,+\,^t\nabla u\cdot\Omega\,+\,
\nabla\left(\frac{1}{\rho}\right)\wedge\nabla\Pi\,=\,0\,,
\end{equation}
where, for two vector-fields $v$ and $w$, we have set $v\wedge w$ to be the skew-symmetric matrix with components
$$
\left(v\wedge w\right)_{ij}\,=\,v^jw^i\,-\,v^iw^j\,.
$$

Finally, recall that we can associate a flow $\psi$ to the velocity field $u$ of the fluid: it is defined by the relation
$$
\psi(t,x)\,\equiv\,\psi_t(x)\,:=\,x\,+\,\int^t_0u(\tau,\psi_\tau(x))\,d\tau
$$
for all $(t,x)\in[0,T]\times\R^N$ and it is, for all fixed $t\in[0,T]$, a diffeomorphism over $\R^N$, if
$\nabla u\in L^\infty$. Let us remark that the flow is still well-defined (in a generalized sense)
even if $u$ is only log-Lipschitz continuous, but it is no more a diffeomorphism (see e.g. chapter 3 of \cite{B-C-D}, or
\cite{C1995}, for more details).

\medskip
Let us now introduce the geometric properties we are handling throughout this paper.
The first notion we are interested in is the \emph{striated regularity}, that is to say initial data are more regular
along some given directions.

So, let us take a family $X=\left(X_\lambda\right)_{1\leq\lambda\leq m}$ of $m$ vector-fields
with components and divergence of class $\mc{C}^\veps$ for some fixed $\veps\in\,]0,1[\,$. We also suppose this family to be
non-degenerate, i.e.
$$
I(X)\,:=\,\inf_{x\in\mbb{R}^N}\,\sup_{\Lambda\in\Lambda^m_{N-1}}\,\left|\stackrel{N-1}{\wedge}X_\Lambda(x)\right|^{\frac{1}{N-1}}
\,\,\,>\,0\,.
$$
Here $\Lambda\in\Lambda^m_{N-1}$ means that $\Lambda=\left(\lambda_1,\ldots,\lambda_{N-1}\right)$, with each
$\lambda_i\in\left\{1,\ldots,m\right\}$ and $\lambda_i<\lambda_j$ for $i<j$, while the symbol
$\stackrel{N-1}{\wedge}X_\Lambda$ stands for the element of $\R^N$ such that
$$
\forall\,\,Y\in\R^N\,,\qquad\left(\stackrel{N-1}{\wedge}X_\Lambda\right)\cdot Y\,=\,
\det\left(X_{\lambda_1}\ldots X_{\lambda_{N-1}},Y\right)\,.
$$
For each vector-field of this family we put
$$
\wtilde{\|}X_\lambda\|_{\mc{C}^\veps}\,:=\,\|X_\lambda\|_{\mc{C}^\veps}\,+\,\|\div X_\lambda\|_{\mc{C}^\veps}\,,
$$
while we will use the symbol $|||\,\cdot\,|||$ in considering the supremum over all indices
$\lambda\in\Lambda^m_1=\{1\ldots m\}$.

\begin{defin} \label{d:stri}
Take a vector-field $Y$ with components and divergence in $\mc{C}^\veps$ and fix a $\eta\in[\veps,1+\veps]$.
A function $f\in L^\infty$ is said to be of class $\mc{C}^\eta$ along $Y$, and we write $\,f\in\mc{C}^\eta_Y$, if
$\div\left(f\,Y\right)\in\,\mc{C}^{\eta-1}\left(\R^N\right)$.

If $X=\left(X_\lambda\right)_{1\leq\lambda\leq m}$ is a non-degenerate family of vector-fields as above, we define
$$
\mc{C}^\eta_X\,:=\,\bigcap_{1\leq\lambda\leq m}\,\mc{C}^\eta_{X_\lambda}\qquad\mbox{ and }
\qquad
\left\|f\right\|_{\mc{C}^\eta_X}\,:=\,\frac{1}{I(X)}\,\left(\|f\|_{L^\infty}\,\wtilde{|||}X|||_{\mc{C}^\veps}\,+\,
|||\,\div\left(f\,X\right)|||_{\mc{C}^{\eta-1}}\right)\,.
$$
\end{defin}

\begin{rem} \label{r:div}
 Our aim is to investigate H\"older regularity of the derivation of $f$ along a fixed vector-field (say) $Y$, i.e. the
quantity
$$
\d_Yf\,:=\,\,\sum_{i=1}^N\,Y^i\,\d_if\,.
$$
If $f$ is only bounded, however, this expression has no meaning: this is why we decided to focus on $\,\div(f\,Y)$, as done
in the literature about this topic (see also \cite{D1999}, section 1). Lemma \ref{l:div}
below will clarify the relation between these two quantities.
\end{rem}


Now, let us take a vector-field $X_0$ and define its time evolution $X(t)$:
\begin{equation} \label{def:X}
X(t,x)\,\equiv\,X_t(x)\,:=\,\d_{X_0(x)}\psi_t\left(\psi^{-1}_t(x)\right)\,,
\end{equation}
that is $X(t)$ is the vector-field $X_0$ transported by the flow associated to $u$.
From this definition, it immediately follows that $\left[X(t)\,,\,\d_t+u\cdot\nabla\right]=0$, i.e. $X(t)$ satisfies the following
system:
\begin{equation} \label{eq:X}
 \left\{\begin{array}{l}
         \left(\d_t\,+\,u\cdot\nabla\right)X\,=\,\d_Xu \\[1ex]
         X_{|t=0}\,=\,X_0\,.
        \end{array}\right.
\end{equation}

We are now ready for stating our first result, on striated regularity.
\begin{theo} \label{t:stri-N}
 Fix $\veps\in\,]0,1[$ and take a non-degenerate family of vector-fields $X_0=\left(X_{0,\lambda}\right)_{1\leq\lambda\leq m}$
over $\R^N$, whose components and divergence are in $\mc{C}^\veps$. \\
Let the initial velocity field $u_0\in L^p$, with $p\in\,]2,+\infty]$, and its vorticity $\Omega_0\in L^\infty\cap L^q$, with
$q\in\,[2,+\infty[$ such that $1/p\,+\,1/q\,\geq\,1/2$. \\
Suppose the initial density $\rho_0\in W^{1,\infty}$ to be such that $0<\rho_*\leq\rho_0\leq\rho^*$. \\
Finally, let us assume that $\Omega_0$ and $\nabla\rho_0$ both belong to $\mc{C}^\veps_{X_0}$.

Then there exist a time $T>0$ and a unique solution $\left(\rho,u,\nabla\Pi\right)$ of system \eqref{eq:ddeuler}, such that:
\begin{itemize}
 \item $\rho\,\in\,L^\infty([0,T];W^{1,\infty})\cap\mc{C}_b([0,T]\times\R^N)$, such that $0<\rho_*\leq\rho\leq\rho^*$
at every time;
 \item $u\,\in\,\mc{C}([0,T];L^p)\cap L^\infty([0,T];\mc{C}^{0,1})$, with
$\d_tu\,\in\,\mc{C}([0,T];L^2)$ and vorticity $\,\Omega\,\in\,\mc{C}([0,T];L^q)$;
 \item $\nabla\Pi\,\in\,\mc{C}([0,T];L^2)$, with $\nabla^2\Pi\,\in\,L^\infty([0,T];L^\infty)$.
\end{itemize}
Moreover, the family of vector-fields transported by the flow still remains, at every time, non-degenerate and with components and
divergence in $\mc{C}^\veps$, and striated regularity is preserved: for all $t\in[0,T]$, one has
$$
\nabla\rho(t)\,,\,\,\Omega(t)\,\,\in\;\mc{C}^{\veps}_{X(t)}\qquad\mbox{ and }\qquad
u(t)\,,\,\,\nabla\Pi(t)\,\,\in\;\mc{C}^{1+\veps}_{X(t)}
$$
uniformly on $[0,T]$.
\end{theo}

Another interesting notion, strictly related to the previous one, is that of \emph{conormal regularity}.
First of all, we have to recall a definition (see again section 1 of \cite{D1999}).
\begin{defin}
Let $\Sigma\subset\mbb{R}^N$ be a compact hypersurface of class $\mc{C}^{1,\veps}$. Let us denote by $\mc{T}^\veps_\Sigma$ the set
of all vector-fields $X$ with components and divergence in $\mc{C}^\veps$, which are tangent to $\Sigma$, i.e.
$\d_XH\,_{|\Sigma}\equiv0$ for all local equations $H$ of $\Sigma$.

Given a $\eta\in[\veps,1+\veps]$, we say that a function $f\in L^\infty$ belongs to the space $\mc{C}^{\eta}_\Sigma$ if
$$
\forall\,X\,\in\,\mc{T}^\veps_\Sigma\;\;,\qquad\div\left(f\,X\right)\,\in\,\mc{C}^{\eta-1}\,.
$$
\end{defin}

Similarly to what happens for striated regularity, also conormal structure propagates during the time evolution.
\begin{theo} \label{t:conorm-N}
 Fix $\veps\in\,]0,1[$ and take a $\mc{C}^{1,\veps}$ compact hypersurface $\Sigma_0\subset\R^N$. \\
Let us suppose the initial velocity field $u_0\in L^p$, with $p\in\,]2,+\infty]$, and its vorticity $\Omega_0\in L^\infty\cap L^q$, with
$q\in\,[2,+\infty[$ such that $1/p\,+\,1/q\,\geq\,1/2$. \\
Assume that the initial density $\rho_0\in W^{1,\infty}$ be such that $0<\rho_*\leq\rho_0\leq\rho^*$. \\
Finally, let $\Omega_0$ and $\nabla\rho_0$ belong to $\mc{C}^{\veps}_{\Sigma_0}$.

Then there exist a time $T>0$ and a unique solution $(\rho,u,\nabla\Pi)$ of system \eqref{eq:ddeuler}, which verifies
the same properties of theorem \ref{t:stri-N}. \\
Moreover, if we define
$$
\Sigma(t)\,:=\,\psi_t\left(\Sigma_0\right)\,,
$$
$\Sigma(t)$ is, at every time $t\in[0,T]$, a hypersurface of class $\mc{C}^{1,\veps}$ of $\R^N$, and conormal regularity is preserved:
at every time $t\in[0,T]$, one has
$$
\nabla\rho(t)\,,\,\,\Omega(t)\,\,\in\;\mc{C}^{\veps}_{\Sigma(t)}\qquad\mbox{ and }\qquad
u(t)\,,\,\,\nabla\Pi(t)\,\,\in\;\mc{C}^{1+\veps}_{\Sigma(t)}
$$
uniformly on $[0,T]$.
\end{theo}

\section{Tools} \label{s:tools}

In this section we will introduce the main tools we used to prove our results; they are mostly based on Fourier analysis techniques. 
Unless otherwise specified, one can find the proof of all the results quoted here in \cite{B-C-D}, chapter 2.

\subsection{Littlewood-Paley decomposition and Besov spaces}

Let us first define the so called ``Littlewood-Paley decomposition'', based on a non-homogeneous dyadic partition of unity with
respect to the Fourier variable.
So, fix a smooth radial function
$\chi$ supported in (say) the ball $B(0,\frac43),$ 
equal to $1$ in a neighborhood of $B(0,\frac34)$
and such that $r\mapsto\chi(r\,e)$ is nonincreasing
over $\R_+$ for all unitary vector $e\in\R^N$. Moreover, set
$\varphi\left(\xi\right)=\chi\left(\xi/2\right)-\chi\left(\xi\right).$
\smallbreak
The dyadic blocks $(\Delta_j)_{j\in\Z}$
 are defined by\footnote{Throughout we agree  that  $f(D)$ stands for 
the pseudo-differential operator $u\mapsto\mc{F}^{-1}(f\,\mc{F}u)$.} 
$$
\Delta_j:=0\ \hbox{ if }\ j\leq-2,\quad\Delta_{-1}:=\chi(D)\quad\hbox{and}\quad
\Delta_j:=\varphi(2^{-j}D)\ \text{ if }\  j\geq0.
$$
We  also introduce the following low frequency cut-off operator:
$$
S_ju:=\chi(2^{-j}D)=\sum_{k\leq j-1}\Delta_{k}\quad\text{for}\quad j\geq0.
$$
The following classical properties will be used freely throughout the paper:
\begin{itemize}
\item for any $u\in\mc{S}',$ the equality $u=\sum_{j}\Delta_ju$ holds true in $\mc{S}'$;
\item for all $u$ and $v$ in $\mc{S}'$,
the sequence $\left(S_{j-1}u\,\,\Delta_jv\right)_{j\in\N}$ is spectrally supported in dyadic annuli.
\end{itemize}
One can now define what a (non-homogeneous) Besov space $B^s_{p,r}$ is.
\begin{defin} \label{d:besov}
  Let  $u$ be a tempered distribution, $s$ a real number, and 
$1\leq p,r\leq+\infty.$ We define the space $B^s_{p,r}$ as the set of distributions $u\in\mc{S}'$ such  that
$$
\|u\|_{B^s_{p,r}}:=\bigg\|\bigl(2^{js}\,
\|\Delta_j  u\|_{L^p}\bigr)_{j\geq-1}\bigg\|_{\ell^r}\;<\;+\infty\,.
$$
\end{defin}
  
From the above definition, it is easy to show that for all $s\in\R$, the Besov space $B^s_{2,2}$ coincides with the non-homogeneous
Sobolev space $H^s$, while for all $s\in\,\R_+\!\!\setminus\!\N$, the space $B^s_{\infty,\infty}$ is actually the
H\"older space $\mc{C}^s$.

If $s\in\N$, instead, we set $\mc{C}^s_*:=B^s_{\infty,\infty}$, to distinguish it from the space $\mc{C}^s$ of
the differentiable functions with continuous partial derivatives up to the order $s$. Moreover, the strict inclusion
$\mc{C}^s_b\,\hra\,\mc{C}^s_*$ holds, where $\mc{C}^s_b$ denotes the subset of $\mc{C}^s$ functions bounded with all
their derivatives up to the order $s$.

If $s<0$, we define the ``negative H\"older space'' $\mc{C}^s$ as the Besov space $B^s_{\infty,\infty}$.

  Finally, let us also point out that for any $k\in\N$ and $p\in[1,+\infty]$, we have the following chain of continuous embeddings:
 $$
 B^k_{p,1}\hookrightarrow W^{k,p}\hookrightarrow B^k_{p,\infty}\,,
 $$
  where  $W^{k,p}$ denotes the set of $L^p$ functions
 with derivatives up to the order $k$ in $L^p$.
\medbreak

Besov spaces have many nice properties which will be recalled throughout the paper
whenever they are needed.
For the time being, let us just mention that if the conditions
$$
s\,>\,1\,+\,\frac{N}{p}\qquad\qquad\mbox{ or }\qquad\qquad s\,=\,1\,+\,\frac{N}{p}\quad\mbox{and}\quad r=1
$$
holds true,
then $B^s_{p,r}$ is an algebra continuously embedded in the set $\mc{C}^{0,1}$
of bounded Lipschitz functions, and that the gradient operator maps $B^s_{p,r}$ in $B^{s-1}_{p,r}$.

\medbreak
The following result will be also needed.
\begin{prop}\label{p:CZ}
Let $F:\R^N\rightarrow\R$
 be a smooth homogeneous function of degree $m$ away from a neighborhood of the origin.

 Then for all $(p,r)\in[1,+\infty]^2$ and all $s\in\R$,  the operator $F(D)$ maps $B^s_{p,r}$
 in $B^{s-m}_{p,r}$.
 \end{prop}
 
  The following fundamental lemma describes, by the so-called \emph{Bernstein's inequalities},
the way derivatives act on spectrally localized functions.
  \begin{lemma} \label{l:bern}
Let  $0<r<R$.   A
constant $C$ exists so that, for any nonnegative integer $k$, any couple $(p,q)$ 
in $[1,+\infty]^2$ with  $1\leq p\leq q$ 
and any function $u\in L^p$,  we  have, for all $\lambda>0$,
$$
\displaylines{
{\rm supp}\, \widehat u \subset   B(0,\lambda R)
\;\Longrightarrow\;
\|\nabla^k u\|_{L^q} \leq
 C^{k+1}\lambda^{k+N\left(\frac{1}{p}-\frac{1}{q}\right)}\|u\|_{L^p}\,;\cr
{\rm supp}\, \widehat u \subset \{\xi\in\R^N\,/\, r\lambda\leq|\xi|\leq R\lambda\}
\;\Longrightarrow\; C^{-k-1}\lambda^k\|u\|_{L^p}
\leq\|\nabla^k u\|_{L^p}\leq C^{k+1}  \lambda^k\|u\|_{L^p}\,.
}$$
\end{lemma}   
  As an immediate consequence of the first Bernstein's inequality, one gets the following embedding result.
  \begin{coroll}\label{c:embed}
  The space $B^{s_1}_{p_1,r_1}$ is continuously embedded in the space $B^{s_2}_{p_2,r_2}$ for all indices satisfying
  $p_1\,\leq\,p_2$ and
  $$
  s_2\,<\,s_1-N\left(\frac{1}{p_1}-\frac{1}{p_2}\right)\qquad\hbox{or}\qquad
  s_2\,=\,s_1-N\left(\frac{1}{p_1}-\frac{1}{p_2}\right)\;\;\hbox{and }\;\;r_1\,\leq\,r_2\,. 
  $$
  \end{coroll}

From now on we will focus on the particular case of H\"older spaces.

\subsection{Paradifferential calculus}

Let us now introduce Bony's decomposition of the product of two tempered distrubutions $u$ and $v$: we will define the paraproduct
operator and recall a few nonlinear estimates in H\"older spaces. Constructing the paraproduct operator relies on the observation
that, formally, the product $u\,v$, may be decomposed into 
\begin{equation}\label{eq:bony}
u\,v\,=\,T_uv\,+\,T_vu\,+\,R(u,v)\,,
\end{equation}
with 
$$
T_uv\,:=\,\sum_jS_{j-1}u\,\Delta_jv\qquad\hbox{ and }\qquad
R(u,v)\,:=\,\sum_j\sum_{|k-j|\leq1}\Delta_ju\,\Delta_{k}v\,.
$$
The above operator $T$ is called ``paraproduct'', whereas
$R$ is called ``remainder''.

\smallbreak
The paraproduct and remainder operators have many nice continuity properties. 
The following ones will be of constant use in this paper.
\begin{prop}\label{p:op}
For any $s\in\R$ and $t>0$, the paraproduct operator 
$T$ maps $L^\infty\times\mc{C}^s$ into $\mc{C}^s$
and  $\mc{C}^{-t}\times\mc{C}^s$ into $\mc{C}^{s-t}$, and the following estimates hold:
$$
\|T_uv\|_{\mc{C}^s}\,\leq\,C\,\|u\|_{L^\infty}\,\|\nabla v\|_{\mc{C}^{s-1}}\qquad\hbox{and}\qquad
\|T_uv\|_{\mc{C}^{s-t}}\,\leq\,C\,\|u\|_{\mc{C}^{-t}}\,\|\nabla v\|_{\mc{C}^{s-1}}\,.
$$

For any $s_1$ and $s_2$ in $\R$ such that  $s_1+s_2>0$,
the remainder operator $R$ maps
$\mc{C}^{s_1}\times\mc{C}^{s_2}$ into $\mc{C}^{s_1+s_2}$ continuously.
\end{prop}
Combining the above proposition with Bony's decomposition \eqref{eq:bony}, 
we easily get the following ``tame estimate'':
\begin{coroll} \label{c:op}
Let $u$ be a bounded function such that $\nabla u\in \mc{C}^{s-1}$ for some $s>0$. 

Then for any $v\in\mc{C}^s$ we have $u\,v\,\in\,\mc{C}^s$
and there exists a constant $C$, depending only on $N$ and $s$,
such that 
$$
\|u\,v\|_{\mc{C}^s}\,\leq\, C\Bigl(\|u\|_{L^\infty}\,\|v\|_{\mc{C}^s}\,+\,\|v\|_{L^\infty}\,\|\nabla u\|_{\mc{C}^{s-1}}\Bigr)\,.
$$
\end{coroll}

In our computations we will often have to handle compositions between a paraproduct operator and a Fourier multiplier. The following
lemma (see the proof e.g. in \cite{D1997}) provides us with estimates for the commutator operator.
\begin{lemma} \label{l:comm}
 Let $m\in\R$, $R>0$ and $f\in\mc{C}^\infty(\R^N)$ be a homogeneous smooth function of degree $m$ out of the ball $B(0,R)$.

Then, there exists a constant $C$, depending only on $R$, such that, for all $s\in\R$ and all $\sigma<1$, one has:
\begin{equation} \label{est:comm}
 \left\|\left[T_u,f(D)\right]v\right\|_{\mc{C}^{s-m+\sigma}}\,\leq\,
\frac{C}{1-\sigma}\,\|\nabla u\|_{\mc{C}^{\sigma-1}}\,\|v\|_{\mc{C}^s}\,.
\end{equation}
\end{lemma}

Let us now quote another result (see \cite{D2010} for the proof of the former part, \cite{D1997} for the proof of the latter),
pertaining to the composition of functions in Besov spaces, which will be of great importance in the sequel.
\begin{prop}\label{p:comp}
\begin{itemize}
 \item [(i)] Let $I$ be an open  interval of $\R$  and $F:I\rightarrow\R$ a smooth function.
 
Then for all compact subset $J\subset I$ and all $s>0$, there exists a constant $C$
such that, for all function $u$ valued in $J$ and with gradient in $\mc{C}^{s-1}$,  we have
$\nabla(F\circ u)\in \mc{C}^{s-1}$ and 
$$
\|\nabla(F\circ u)\|_{\mc{C}^{s-1}}\,\leq\,C\,\|\nabla u\|_{\mc{C}^{s-1}}\,.
$$

 \item [(ii)] Let $s>0$ and $m\in\N$ be such that $m>s$. Let $u\in \mc{C}^s$ and $\psi\in\mc{C}^m_b$ such that
the Jacobian of $\psi^{-1}$ is bounded.

Then $u\circ\psi\,\in\,\mc{C}^s$. Moreover, if $s\in\,]0,1[\,$ the following estimate holds:
$$
 \left\|u\circ\psi\right\|_{\mc{C}^s}\,\leq\,C\,\left(1+\left\|\nabla\psi\right\|_{L^\infty}\right)\,\|u\|_{\mc{C}^s}\,.
$$
\end{itemize}
\end{prop}

Finally, let us introduce the notion of paravector-field.
\begin{defin} \label{d:pvec-f}
 Let $X$ be a vector-field with coefficients in $\mc{S}'$. We can formally define the paravector-field operator $T_X$ in the
following way: for all $u\in\mc{S}'$,
$$
T_Xu\,:=\,\sum_{i=1}^N\,T_{X^i}\d_iu\,.
$$
\end{defin}
The following result (see \cite{D1999}, section 2 for the proof) says that the paravector-field operator is, in a certain sense, the
principal part of the derivation $\d_X$.
\begin{lemma} \label{l:T_X}
 For all vector field $X\in\mc{C}^s$ and all $u\in\mc{C}^t$, we have:
\begin{itemize}
 \item if $\,t<1$ and $s+t>1$, then
$$
\left\|\d_Xu\,-\,T_Xu\right\|_{\mc{C}^{s+t-1}}\,\leq\,\frac{C}{(1-t)\,(s+t-1)}\,\|X\|_{\mc{C}^s}\,\|\nabla u\|_{\mc{C}^{t-1}}\,;
$$
\item if $\,t<0$, $s<1$ and $s+t>0$, then
$$
\left\|T_Xu\,-\,\div\left(u\,X\right)\right\|_{\mc{C}^{s+t-1}}\,\leq\,\frac{C}{t\,(s+t)\,(s-1)}\,\|X\|_{\mc{C}^s}\,\|u\|_{\mc{C}^t}\,;
$$
\item if $\,t<1$ and $s+t>0$, then
$$
\left\|\d_Xu\,-\,T_Xu\right\|_{\mc{C}^{s+t-1}}\,\leq\,\frac{C}{(s+t)\,(1-t)}\,\wtilde{\|}X\|_{\mc{C}^s}\,\|\nabla u\|_{\mc{C}^{t-1}}\,.
$$
\end{itemize}
Moreover, first and last inequalities are still true even in the case $\,t=1$, provided that one replaces $\|\nabla u\|_{\mc{C}^0_*}$
with $\|\nabla u\|_{L^\infty}$, while the second is still true even if $\,t=0$, with $\|u\|_{L^\infty}$ instead of $\|u\|_{\mc{C}^0_*}$.
\end{lemma}

We will heavily use also the following statement about composition of paravector-field and paraproduct operators
(see the appendix in \cite{D1999} for its proof).
\begin{lemma} \label{l:pvec-pprod}
 Fix $s\in\,]0,1[$. There exist constants $C$, depending only on $s$, such that, for all $t_1<0$ and $t_2\in\R$,
\begin{eqnarray*}
 \left\|T_X\,T_u\,v\right\|_{\mc{C}^{s-1+t_1+t_2}} & \leq & C\bigl(\|X\|_{\mc{C}^s}\,\|u\|_{\mc{C}^{t_1}}\,
\|v\|_{\mc{C}^{t_2}}\,+ \\
& & \qquad+\,\|v\|_{\mc{C}^{t_2}}\,\|T_Xu\|_{\mc{C}^{s-1+t_1}}\,+\,\|u\|_{\mc{C}^{t_1}}\,\|T_Xv\|_{\mc{C}^{s-1+t_2}}\bigr)\,,
\end{eqnarray*}
and this is still true in the case $t_1=0$ with $\|u\|_{L^\infty}$ instead of $\|u\|_{\mc{C}^0_*}$.

Moreover, if $\,s-1+t_1+t_2\,>\,0$, then we have also
\begin{eqnarray*}
 \left\|T_X\,R(u,v)\right\|_{\mc{C}^{s-1+t_1+t_2}} & \leq & C\bigl(\|X\|_{\mc{C}^s}\,\|u\|_{\mc{C}^{t_1}}\,
\|v\|_{\mc{C}^{t_2}}\,+ \\
& & \qquad+\,\|v\|_{\mc{C}^{t_2}}\,\|T_Xu\|_{\mc{C}^{s-1+t_1}}\,+\,\|u\|_{\mc{C}^{t_1}}\,\|T_Xv\|_{\mc{C}^{s-1+t_2}}\bigr)\,.
\end{eqnarray*}
\end{lemma}

\subsection{Transport and elliptic equations}

System \eqref{eq:ddeuler} is basically a coupling of transport equations of the type
$$
\left\{\begin{array}{l}
\d_tf\,+\,v\cdot\nabla f\,=\,g\,,\\[1ex]
f_{|t=0}\,=\,f_0\,.
\end{array}\right.
\leqno(T)
$$
So, we often need to use the following result, which enables us to solve $(T)$ in the H\"older spaces framework.
\begin{prop}\label{p:transport}
Let $\sigma>0$ ($\sigma>-1$ if $\,\div v=0$).\\
Let $f_0\in\mc{C}^\sigma$,
 $g\in L^1([0,T];\mc{C}^\sigma)$ and  $v$
 be a time dependent vector field in $\mc{C}_b([0,T]\times\R^N)$ such that   
$$
\begin{array}{lllll}
\nabla v&\in&  L^1([0,T];L^\infty) &\hbox{\rm if} &\sigma<1\,,\\[2ex]
\nabla v&\in& L^1([0,T];\mc{C}^{\sigma-1}) &\hbox{\rm if} &\sigma>1\,.
\end{array}
$$

Then equation $(T)$ has a unique solution  $f$ in the space
$\Bigl(\bigcap_{\sigma'<\sigma}\mc{C}([0,T];\mc{C}^{\sigma'})\Bigr)\bigcap\mc{C}_w([0,T];\mc{C}^{\sigma})$.

Moreover,  for all $t\in[0,T]$ we have
\begin{equation}\label{est:no-loss-1}
e^{-CV(t)}\,\|f(t)\|_{\mc{C}^\sigma}\,\leq\,
\|f_0\|_{\mc{C}^\sigma}\,+\,\int_0^t e^{-CV(\tau)}\,\|g(\tau)\|_{\mc{C}^\sigma}\,d\tau
\end{equation} 
$$\displaylines{
\mbox{with}\quad\!\! V'(t):=\left\{
\begin{array}{ll}
\!\|\nabla v(t)\|_{L^\infty}\! & \mbox{if}\!\!\!\quad\sigma<1\,,\\[2ex]
\!\|\nabla v(t)\|_{\mc{C}^{\sigma-1}}\! & \mbox{if}\!\!\!\quad\sigma>1\,. 
\end{array}\right.\hfill} $$
 If $f\equiv v$ then, for all $\sigma>0$ ($\sigma>-1$ if $\,\div v=0$),
    estimate  \eqref{est:no-loss-1} holds with
$V'(t):=\|\nabla f(t)\|_{L^\infty}$.
\end{prop}

Finally, we shall make an extensive use of energy  estimates  for the
following elliptic equation:
\begin{equation}\label{eq:elliptic}
-\,\div(a\,\nabla\Pi)=\div F\quad\hbox{in }\ \R^N\,,
\end{equation}
where $a=a(x)$ is a given suitably smooth bounded function satisfying 
\begin{equation}\label{eq:ellipticity}
a_*\,:=\,\inf_{x\in\R^N}a(x)\,>\,0\,.
\end{equation}

We shall use  the following  result
based on Lax-Milgram's theorem (see the proof in e.g. \cite{D2010}).  
\begin{lemma}\label{l:laxmilgram}
For all vector field $F$ with coefficients in $L^2$, there exists a tempered distribution $\Pi$,
unique up to  constant functions, 
such that  $\nabla\Pi\in L^2$ and  
equation $\eqref{eq:elliptic}$ is satisfied. 
In addition, we have 
$$
a_*\,\|\nabla\Pi\|_{L^2}\,\leq\,\|F\|_{L^2}\,.
$$
\end{lemma}

\section{Propagation of striated regularity} \label{s:striated}

Now we are ready to tackle the proof of theorem \ref{t:stri-N}. We will carry out it in a standard way: first of all we will
prove a priori estimates for smooth solutions to \eqref{eq:ddeuler}. Then, we will construct a sequence of regular
approximated solutions. Finally, thanks to the the just proved upper bounds, we will get convergence of this sequence to a
solution of our initial system, with the required properties.

\subsection{A priori estimates}

First of all, we will prove a priori estimates for a smooth solution $(\rho,u,\nabla\Pi)$ to system \eqref{eq:ddeuler}.

\subsubsection{Estimates for the density and the velocity field}

From first equation of \eqref{eq:ddeuler}, it follows that
$$
\rho(t,x)\,=\,\rho_0\left(\psi^{-1}_t(x)\right)\,,
$$
so, as the flow $\psi_t$ is a diffeomorphism over $\R^N$ at all fixed time, we have that
\begin{equation}
 0\,<\,\rho_*\,\leq\,\rho(t)\,\leq\,\rho^* \label{est:rho_L^inf}\,.
\end{equation}
Applying the operator $\,\d_i$ to the same equation, using classical $L^p$ estimates for the transport equation and Gronwall's lemma,
we get
\begin{equation} \label{est:Drho_L^inf}
 \|\nabla\rho(t)\|_{L^\infty}\,\leq\,\|\nabla\rho_0\|_{L^\infty}\,\exp\left(C\int^t_0\|\nabla u\|_{L^\infty}\,d\tau\right)\,.
\end{equation}

From the equation for the velocity, instead, we get in the same way
$$
 \|u(t)\|_{L^p}\,\leq\,\|u_0\|_{L^p}\,+\,\int^t_0\left\|\frac{\nabla\Pi}{\rho}\right\|_{L^p}\,d\tau\,;
$$
so, using \eqref{est:rho_L^inf} and H\"older inequalities, the following estimate holds for some $\theta\in\,]0,1[\,$:
\begin{equation} \label{est:u_L^p}
 \|u(t)\|_{L^p}\,\leq\,\|u_0\|_{L^p}\,+\,\frac{C}{\rho_*}\,\int^t_0\|\nabla\Pi\|^\theta_{L^2}\,
\|\nabla\Pi\|^{1-\theta}_{L^\infty}\,d\tau\,.
\end{equation}

\begin{rem}
 Let us observe that, as regularity of the pressure goes like that of the velocity field, one can try to estimate directly
the $L^p$ norm of the pressure term. Unfortunately, we can't solve its (elliptic) equation in this space without assuming a
smallness condition on the gradient of the density. So, we will prove that $\nabla\Pi$ is in $L^2\cap L^\infty$,
which is actually stronger than the previous property and requires no other hypothesis on the density term.
\end{rem}

From \eqref{est:Drho_L^inf} it's clear that we need an estimate for the $L^\infty$ norm for the gradient of the velocity.
As remarked before, we can't expect to get it from the hypothesis $\Omega\in L^\infty$; the key will be the further assumption
of more regularity of the vorticity along the directions given by the family $X_0$. \\
Here we quote also a fundamental lemma, whose proof can be found in \cite{B-C-D} (chapter 7) for the $2$-dimensional
case, in \cite{D1999} (section 3) and \cite{G-SR} (again section 3) for the general one. It is the key point to get the
velocity field to be Lipschitz and it turns out to be immediately useful in the sequel.

\begin{lemma} \label{l:Du_L^inf}
 Fix $\veps\in\,]0,1[$ and an integer $m\geq N-1$ and take a non-degenerate family $Y=\left(Y_\lambda\right)_{1\leq\lambda\leq m}$ of
$\mc{C}^\veps$ vector-fields over $\R^N$ such that also their divergences are in $\,\mc{C}^\veps$.

Then, for all indices $1\leq i,j\leq N$, there exist $\mc{C}^\veps$ functions $a_{ij}$, $b^{k\lambda}_{ij}$ (with $1\leq k\leq N$,
$1\leq\lambda\leq m$)  such that, for all $(x,\xi)\in\R^N\times\R^N$, the following equality holds:
$$
\xi_i\,\xi_j\,=\,a_{ij}(x)|\xi|^2\,+\,\sum_{k,\lambda}b^{k\lambda}_{ij}(x)\left(Y_\lambda(x)\cdot\xi\right)\xi_k\,.
$$
Moreover, the functions in the previous relation could be chosen such that
$$
 \left\|a_{ij}\right\|_{L^\infty}\;\leq\;1\qquad\mbox{ and }\qquad
\left\|b^{k\lambda}_{ij}\right\|_{\mc{C}^\veps}\;\leq\;C\,\frac{m^{2N-2}}{I(Y)}\,\,|||Y|||\,^{9N-10}_{\mc{C}^\veps}\,.
$$
\end{lemma}

Now, we can state the stationary estimate which says that the velocity field $u$ is Lipschitz. This can be done as in
the classical case, because it's based only on the Biot-Savart law.
\begin{prop} \label{p:Du_L^inf}
 Fix $\veps\in\,]0,1[$ and $q\in\,]1,+\infty[$ and take a non-degenerate family $Y=\left(Y_\lambda\right)_{1\leq\lambda\leq m}$
of $\mc{C}^\veps$ vector-fields over $\R^N$ such that also their divergences are still in $\mc{C}^\veps$.

Then there exists a constant $C$, depending only on the space dimension $N$ and on the number of vector-fields $m$, such that, for all skew-symmetric
matrices $\Omega$ with coefficients in $L^q\cap\mc{C}^\veps_Y$, the corresponding (by \eqref{eq:BS-law}) divergence-free vector-field $u$ satisfies
\begin{equation}
 \left\|\nabla u\right\|_{L^\infty}\,\leq\,C\,\left(\frac{q^2}{q-1}\,\|\Omega\|_{L^q}\,+\,\frac{1}{\veps\,(1-\veps)}\,\|\Omega\|_{L^\infty}
\log\left(e\,+\,\frac{\left\|\Omega\right\|_{\mc{C}^\veps_Y}}{\|\Omega\|_{L^\infty}}\right)\right)\,. \label{est:Du_L^inf}
\end{equation}
\end{prop}

\subsubsection{Estimates for the vorticity}

As in \cite{D-F}, using the well-known $L^q$ estimates for transport equation and taking advantage of Gronwall's lemma and H\"older
inequality in Lebesgue spaces, from \eqref{eq:vort} we obtain, for some $\gamma\in\,]0,1[\,$,
\begin{eqnarray}
 \|\Omega(t)\|_{L^q} & \leq & C\,\exp\left(\int^t_0\|\nabla u\|_{L^\infty}d\tau\right)\times \label{est:Om_L^q} \\
& & \times\left(\|\Omega_0\|_{L^q}\,+\,\frac{1}{\left(\rho_*\right)^2}\int^t_0e^{-\int^\tau_0\|\nabla u\|_{L^\infty}d\tau'}
\left\|\nabla\rho\right\|_{L^\infty}\,\left\|\nabla\Pi\right\|^\gamma_{L^2}\,\left\|\nabla\Pi\right\|^{1-\gamma}_{L^\infty}
d\tau\right). \nonumber
\end{eqnarray}

Moreover, of course an analogue estimate holds also for the $L^\infty$ norm:
\begin{eqnarray}
 \|\Omega(t)\|_{L^\infty} & \leq & C\,\exp\left(\int^t_0\|\nabla u\|_{L^\infty}d\tau\right)\times \label{est:Om_L^inf} \\
& & \times\left(\|\Omega_0\|_{L^\infty}\,+\,\frac{1}{\left(\rho_*\right)^2}\int^t_0e^{-\int^\tau_0\|\nabla u\|_{L^\infty}d\tau'}
\left\|\nabla\rho\right\|_{L^\infty}\,\left\|\nabla\Pi\right\|_{L^\infty}\,d\tau\right)\,. \nonumber
\end{eqnarray}

\begin{rem} \label{r:q}
 Let us fix the index $p$ pertaining to $u$ and let us call $\oline{q}$ the real number in $[2,+\infty[$ such that
$1/p\,+\,1/\oline{q}\,=\,1/2$. From our hypothesis, it's clear that $q\leq\oline{q}$; therefore, thanks to
H\"older and Young inequalities, we have
$$
\left\|\Omega\right\|_{L^{\oline{q}}}\,\leq\,\left\|\Omega\right\|^\eta_{L^q}\,\left\|\Omega\right\|^{1-\eta}_{L^\infty}\,\leq\,
\left\|\Omega\right\|_{L^q\cap L^\infty}\,.
$$
\end{rem}

\subsubsection{Estimates for the pressure term} \label{sss:est-p}

Now, let us focus on the pressure term: taking the divergence of the second equation of system \eqref{eq:ddeuler}, we discover that
it solves the elliptic equation
\begin{equation} \label{eq:Pi}
 -\,\div\left(\frac{\nabla\Pi}{\rho}\right)\,=\,\div\left(u\cdot\nabla u\right)\,.
\end{equation}
From this, remembering our hypothesis and remark \ref{r:q},
estimate \eqref{est:CZ} and lemma \ref{l:laxmilgram}, the control of $L^2$ norm immeditely follows:
\begin{equation} \label{est:Pi_L^2}
 \frac{1}{\rho^*}\,\|\nabla\Pi\|_{L^2}\,\leq\,C\,\|u\|_{L^p}\,\|\Omega\|_{L^q\cap L^\infty}\,.
\end{equation}

Moreover, we have that $\nabla\Pi$ belongs also to $L^\infty$, and so $\nabla\Pi\in L^a$ for all $a\in[2,+\infty]$.
In fact, we are going to show a stronger claim, that is to say $\nabla\Pi\in\mc{C}^1_*$. Cutting in low and
high frequencies, we have that
$$
\|\nabla\Pi\|_{\mc{C}^1_*}\,\leq\,\|\Delta_{-1}\nabla\Pi\|_{\mc{C}^1_*}\,+\,\|(\Id-\Delta_{-1})\nabla\Pi\|_{\mc{C}^1_*}\,\leq\,
C\left(\|\nabla\Pi\|_{L^2}\,+\,\|\Delta\Pi\|_{\mc{C}^0_*}\right)\,.
$$
Now, from \eqref{eq:Pi} we obtain
\begin{equation} \label{eq:Lapl-Pi}
 -\,\Delta\Pi\,=\,\nabla\left(\log\rho\right)\cdot\nabla\Pi\,+\,\rho\,\,\div\left(u\cdot\nabla u\right)\,.
\end{equation}
From this relation, the fact that $\div(u\cdot\nabla u)=\nabla u:\nabla u$ and the immersion
$L^\infty\hookrightarrow\mc{C}^0_*$, we get
\begin{eqnarray*}
\|\Delta\Pi\|_{\mc{C}^0_*}\;\;\leq\;\;\|\Delta\Pi\|_{L^\infty} & \leq &
\left\|\nabla\left(\log\rho\right)\cdot\nabla\Pi\right\|_{L^\infty}\,+\,
\left\|\rho\,\,\div\left(u\cdot\nabla u\right)\right\|_{L^\infty} \\
 & \leq & C\left(\left\|\nabla\rho\right\|_{L^\infty}\,\|\nabla\Pi\|_{L^\infty}\,+\,\rho^*\,
\left\|\nabla u\right\|^2_{L^\infty}\right)\,.
\end{eqnarray*}
Now, by interpolation we have, for some $\beta\in\,]0,1[\,$ depending only on the dimension $N$,
$$
\|\nabla\Pi\|_{L^\infty}\,\leq\,C\,\|\nabla\Pi\|^\beta_{\mc{C}^{-N/2}}\,\|\nabla\Pi\|^{1-\beta}_{\mc{C}^1_*}\,\leq\,
C\,\|\nabla\Pi\|^\beta_{L^2}\,\|\nabla\Pi\|^{1-\beta}_{\mc{C}^1_*}\,.
$$
Thanks to Young's inequality, from this relation and \eqref{est:Pi_L^2} one finally gets
\begin{equation} \label{est:Pi_C^1_*}
 \left\|\nabla\Pi\right\|_{\mc{C}^1_*}\,\leq\,C\,\left(\left(1+\|\nabla\rho\|^\delta_{L^\infty}\right)
\|u\|_{L^p}\,\|\Omega\|_{L^q\cap L^\infty}\,+\,
\rho^*\,\left\|\nabla u\right\|^2_{L^\infty}\right)\,,
\end{equation}
for some $\delta$ depending only on $\beta$, and so finally on the space dimension $N$. So our claim is now proved.

Finally, we want to find a bound on the second derivatives of the pressure term. For doing this,
we will need the striated norm of $\nabla\Pi$. In fact, passing in Fourier variables and using
lemma \ref{l:Du_L^inf}, for all $1\leq i\,,j\leq N$ we can write
$$
\xi_i\,\xi_j\,\what{\Pi}(\xi)\,=\,a_{ij}(x)|\xi|^2\,\what{\Pi}(\xi)\,+\,
\sum_{k,\lambda}b^{k\lambda}_{ij}(x)\left(X_\lambda(x)\cdot\xi\right)\xi_k\,\what{\Pi}(\xi)\,.
$$
Applying the inverse Fourier transform $\mc{F}^{-1}_\xi$ and passing to $L^\infty$ norms, we get
$$
 \left\|\nabla^2\Pi\right\|_{L^\infty}\,\leq\,C\left(\left\|\Delta\Pi\right\|_{L^\infty}\,+\,
\left\|\d_X\nabla\Pi\right\|_{L^\infty}\right)\,.
$$
Proposition 2.104 of \cite{B-C-D} tells us that
$$
\left\|\d_X\nabla\Pi\right\|_{L^\infty}\,\leq\,\frac{C}{\veps}\,\left\|\d_X\nabla\Pi\right\|_{\mc{C}^0_*}\,
\log\left(e\,+\,\frac{\left\|\d_X\nabla\Pi\right\|_{\mc{C}^\veps}}{\left\|\d_X\nabla\Pi\right\|_{\mc{C}^0_*}}\right)\,.
$$
Using Bony's paraproduct decomposition to handle the norm in $\mc{C}^0_*$ and noticing that
the function $\zeta\mapsto\zeta\log(e+k/\zeta)$ is nondecreasing, we finally get
\begin{eqnarray} 
& &  \left\|\nabla^2\Pi\right\|_{L^\infty}\,\leq\, C\Biggl(\left\|\nabla\rho\right\|_{L^\infty}\,\|\nabla\Pi\|_{\mc{C}^1_*}\,+\,
\rho^*\,\left\|\nabla u\right\|^2_{L^\infty}\,+ \label{est:D^2-Pi} \\
& & \qquad\qquad\qquad\qquad\qquad\qquad+\,\wtilde{\|}X\|_{\mc{C}^\veps}\,\|\nabla\Pi\|_{\mc{C}^1_*}\,
\log\left(e+\frac{\left\|\d_X\nabla\Pi\right\|_{\mc{C}^\veps}}{\wtilde{\|}X\|_{\mc{C}^\veps}\|\nabla\Pi\|_{\mc{C}^1_*}}\right)
\Biggr). \nonumber
\end{eqnarray}

\subsection{A priori estimates for striated regularity}

After having established the ``classical'' estimates, let us now focus on the conservation of striated regularity.
The most important step lies in finding a priori estimates for the derivations along the vector-field $X$.
So, let us now state a lemma which explains the relation between the operators $\d_X$ and $\,\div(\,\cdot\,X)$ (see also
remark \ref{r:div}).

\begin{lemma} \label{l:div}
 For every vector-field $X$ with components and divergence in $\mc{C}^\veps$, and every function $f\in\mc{C}^\eta$ for some
$\eta\in\,]0,1]$, we have
$$
\left\|\div(f\,X)\,-\,\d_Xf\right\|_{\mc{C}^{\min\{\veps,\eta\}}}\,\leq\,C\,\,\wtilde{\|}X\|_{\mc{C}^\veps}\,\|f\|_{\mc{C}^\eta}\,.
$$
Moreover, the previous inequality is still true in the limit case $\eta=0$, with $\|\cdot\|_{L^\infty}$ instead of
$\|\cdot\|_{\mc{C}^0_*}$.
\end{lemma}

\begin{proof}
 The claim immediately follows from the identity $\,\div(f\,X)\,-\,\d_Xf\,=\,f\,\div\!X$ and from Bony's paraproduct decomposition.
\end{proof}

\subsubsection{The evolution of the family of vector-fields}

First of all, we want to prove that the family of vector-fields $X(t)=\left(X_\lambda(t)\right)_{1\leq\lambda\leq m}$, where
each $X_\lambda(t)$ is defined by \eqref{def:X}, still remains non-degenerate
for all $t$, and that each $X_\lambda(t)$ still has components and divergence in $\mc{C}^\veps$.
Throughout this paragraph we will denote by $Y(t)$ a generic element of the family $X(t)$.

Applying the divergence operator to \eqref{eq:X}, an easy computation shows us that $\div Y$ satisfies
$$
\left(\d_t\,+\,u\cdot\nabla\right)\,\div Y\,=\,0\,,
$$
which immediately implies $\div Y(t)\in\mc{C}^\veps$ for all $t$ and
\begin{equation} \label{est:div-X}
 \left\|\div Y(t)\right\|_{\mc{C}^\veps}\,\leq\,C\,\left\|\div Y_0\right\|_{\mc{C}^\veps}
\exp\left(c\,\int^t_0\|\nabla u\|_{L^\infty}\,d\tau\right)\,.
\end{equation}

Moreover, starting again from \eqref{eq:X}, we get (for the details, see proposition 4.1 of \cite{D1999})
$$
\left(\d_t\,+\,u\cdot\nabla\right)\left(\stackrel{N-1}{\wedge}X_\lambda\right)\,=\,
^t\nabla u\cdot\left(\stackrel{N-1}{\wedge}X_\lambda\right)\,,
$$ 
from which it follows
$$
\left(\stackrel{N-1}{\wedge}X_\lambda\right)(t,x)\,=\,\left(\stackrel{N-1}{\wedge}X_\lambda\right)(0,\psi^{-1}_t(x))\,-\,
\int^t_0\,^t\nabla u\cdot\left(\stackrel{N-1}{\wedge}X_\lambda\right)(\tau,\psi^{-1}_t(\psi_\tau(x)))\,d\tau\,.
$$
This relation gives us
\begin{eqnarray*}
 \left|\left(\stackrel{N-1}{\wedge}X_\lambda\right)(0,\psi^{-1}_t(x))\right| & \leq &
\left|\left(\stackrel{N-1}{\wedge}X_\lambda\right)(t,x)\right|\,+\, \\
& & +\,\int^t_0\left\|\nabla u(t-\tau)\right\|_{L^\infty}\,
\left|\left(\stackrel{N-1}{\wedge}X_\lambda\right)(t-\tau,\psi^{-1}_\tau(x))\right|\,d\tau\,,
\end{eqnarray*}
and by Gronwall's lemma one gets
$$
\left|\left(\stackrel{N-1}{\wedge}X_\lambda\right)(t,x)\right|\,\geq\,\left|\left(\stackrel{N-1}{\wedge}X_{0,\lambda}\right)(\psi^{-1}_t(x))\right|\,
e^{-c\,\int^t_0\|\nabla u\|_{L^\infty}\,d\tau}\,.
$$
Therefore the family still remains non-degenerate at every time $t$:
\begin{equation} \label{est:I}
 I(X(t))\,\geq\,I(X_0)\,\exp\left(-\,c\int^t_0\|\nabla u\|_{L^\infty}\,d\tau\right)\,.
\end{equation}

Finally, again from the evolution equation \eqref{eq:X}, it's clear that, to prove that $Y(t)$ is of class $\mc{C}^\veps$, we
need a control on the norm in this space of the term $\d_Yu$. To get this, we use, as very often in the sequel,
the paravector-field decomposition:
$$
\d_Yu\,=\,T_Yu\,+\,\left(\d_Y-T_Y\right)u\,,
$$
with (by lemma \ref{l:T_X})
$$
\left\|\left(\d_Y-T_Y\right)u\right\|_{\mc{C}^\veps}\,\leq\,C\,\,\wtilde{\|}Y\|_{\mc{C}^\veps}\,\|\nabla u\|_{L^\infty}\,.
$$
Moreover, for all $1\leq i\leq N$ thanks to \eqref{eq:BS-law} we can write
$$
T_Yu^i\,=\,-\,\sum_{k,j}\left(\d_k\left(-\Delta\right)^{-1}T_{Y^j}\d_j\Omega_{ik}\,-\,
\left[\d_k\left(-\Delta\right)^{-1},T_{Y^j}\d_j\right]\Omega_{ik}\right)\,.
$$
Obviously, by lemma \ref{l:T_X} again we have
$$
\left\|\d_k\left(-\Delta\right)^{-1}\sum_jT_{Y^j}\d_j\Omega_{ik}\right\|_{\mc{C}^\veps}\,\leq\,
\left\|T_Y\Omega\right\|_{\mc{C}^{\veps-1}}\,\leq\,
\|\d_Y\Omega\|_{\mc{C}^{\veps-1}}\,+\,C\,\wtilde{\|}Y\|_{\mc{C}^\veps}\,\|\Omega\|_{L^\infty}\,,
$$
while for the commutator term we use lemma \ref{l:comm}, which gives us the following control:
$$
\left\|\left[\d_k\left(-\Delta\right)^{-1},T_{Y^j}\d_j\right]\Omega_{ik}\right\|_{\mc{C}^\veps}\,\leq\,C\,\|Y\|_{\mc{C}^\veps}\,
\|\Omega\|_{L^\infty}\,.
$$
So, in the end, from the hypothesis of striated regularity for the vorticity we get that also the velocity field $u$ is more regular
along the fixed directions and
\begin{equation} \label{est:d_X-u}
 \left\|\d_Yu\right\|_{\mc{C}^\veps}\,\leq\,C\left(\|\d_Y\Omega\|_{\mc{C}^{\veps-1}}\,+\,
\wtilde{\|}Y\|_{\mc{C}^\veps}\,\|\nabla u\|_{L^\infty}\right)\,.
\end{equation}
Moreover, applying proposition \ref{p:transport} to \eqref{eq:X} and using \eqref{est:d_X-u}, \eqref{est:div-X} and Gronwall's
inequality finally give us
\begin{equation} \label{est:X_C^e}
 \wtilde{\|}Y(t)\|_{\mc{C}^\veps}\,\leq\,C\,\exp\left(c\int^t_0\|\nabla u\|_{L^\infty}\,d\tau\right)
\left(\wtilde{\|}Y_0\|_{\mc{C}^\veps}\,+\,
\int^t_0e^{-\,c\int^\tau_0\|\nabla u\|_{L^\infty}\,d\tau'}\,\left\|\d_Y\Omega\right\|_{\mc{C}^{\veps-1}}\,d\tau\right).
\end{equation}

These estimates having being established, from now on for simplicity we will consider the case of only one vector-field $X(t)$:
the generalization  to the case of a finite family is quite obvious, and where the difference is substantial, we will suggest
references for the details.

\subsubsection{Striated regularity for the density}

Now, we want to investigate propagation of striated regularity for the density. First of all, let us state a stationary lemma.
\begin{lemma} \label{l:f->Df}
 Let $f$ be a function in $\mc{C}^1_*$.
\begin{itemize}
 \item[(i)] If $\,\d_Xf\in\mc{C}^\veps$ and $\nabla f\in L^\infty$, then one has $\,\d_X\nabla f\in\mc{C}^{\veps-1}$ and
the following inequality holds:
\begin{equation} \label{est:f->Df}
\left\|\d_X\nabla f\right\|_{\mc{C}^{\veps-1}}\,\leq\,C\left(\left\|\d_Xf\right\|_{\mc{C}^\veps}\,+\,
\wtilde{\|}X\|_{\mc{C}^\veps}\left(\|f\|_{\mc{C}^1_*}\,+\,\|\nabla f\|_{L^\infty}\right)\right)\,.
\end{equation}
\item[(ii)] Conversely, if $\,\d_X\nabla f\in\mc{C}^{\veps-1}$, then $\,\d_Xf\in\mc{C}^\veps$ and one has
\begin{equation} \label{est:Df->f}
\left\|\d_Xf\right\|_{\mc{C}^\veps}\,\leq\,C\left(\,\wtilde{\|}X\|_{\mc{C}^\veps}
\left(\|f\|_{\mc{C}^1_*}\,+\,\|\nabla f\|_{L^\infty}\right)\,+\,\left\|\d_X\nabla f\right\|_{\mc{C}^{\veps-1}}\right)\,.
\end{equation}
\end{itemize}
\end{lemma}

\begin{proof}
\begin{itemize}
 \item[(i)] Using the paravector-field operator (remember definition \ref{d:pvec-f}), we can write:
$$
 \d_X\,\nabla f\,=\,\left(\d_X\,-\,T_X\right)\nabla f\,+\,T_X\,\nabla f\,.
$$
From lemma \ref{l:T_X}, we have that the first term of the previous equality is in $\mc{C}^{\veps-1}$ and
\begin{equation} \label{est:d-T_D}
 \left\|\left(\d_X-T_X\right)\nabla f\right\|_{\mc{C}^{\veps-1}}\,\leq\,
C\,\wtilde{\|}X\|_{\mc{C}^\veps}\,\|\nabla f\|_{L^\infty}\,.
\end{equation}
Now, we have to estimate the paravector-field term: note that
$$
T_X\,\nabla f\,=\,\nabla\left(T_Xf\right)\,+\,\left[T_X,\nabla\right]f\,.
$$
From the hypothesis, it's obvious that $\nabla\left(T_Xf\right)\,\in\mc{C}^{\veps-1}$. For the last term,
remembering that $\nabla$ and $T_X$ are operators of order $1$, we can use lemma \ref{l:comm} and get
\begin{equation} \label{est:T_D}
 \left\|\left[T_X,\nabla\right]f\right\|_{\mc{C}^{\veps-1}}\,\leq\,C\,\|X\|_{\mc{C}^\veps}\,\|f\|_{\mc{C}^1_*}\,.
\end{equation}
Putting together \eqref{est:d-T_D}, \eqref{est:T_D} and the control for $\left\|\nabla\left(T_Xf\right)\right\|_{\mc{C}^{\veps-1}}$
gives us the first part of the lemma.
\item[(ii)] For the second part, we write once again $\d_Xf\,=\,T_Xf\,+\,\left(\d_X-T_X\right)f$. \\
By definition of the space $\mc{C}^{\veps}_X$, we know that $\nabla f$ is bounded: so, the latter term can be easily controlled
in $\mc{C}^\veps$ thanks to lemma \ref{l:T_X}. Now let us define the operator $\Psi$ such that, in Fourier variables,
for all vector-fields $v$ we have
$$
\mc{F}_x\left(\Psi v\right)(\xi)\,=\,-\,i\,\frac{1}{|\xi|^2}\,\xi\cdot\what{v}(\xi)\,.
$$
So, noting that the paravector term involves only the high frequencies of $f$, we can write
$$
T_Xf\,=\,T_X\left(\Psi\,\nabla f\right)\,=\,\Psi\,T_X\nabla f\,+\,
\left[T_X,\Psi\right]\nabla f\,.
$$
Now, applying lemmas \ref{l:T_X} and \ref{l:comm} completes the proof.
\end{itemize}
\end{proof}

\begin{rem} \label{r:lem_f->Df}
Let us note that, if $f\in L^a$ (for some $a\in[1,+\infty]$) is such that $\nabla f\in L^\infty$, then $f\in\mc{C}^1_*$
(indeed $f\in\mc{C}^{0,1}$) and (separating low and high frequencies)
$$
\|f\|_{\mc{C}^1_*}\,\leq\,C\,\left(\|f\|_{L^a}\,+\,\left\|\nabla f\right\|_{L^\infty}\right)\,.
$$
Both $u$ and $\rho$ satisfy such an estimate, respectively with $a=p$ and $a=+\infty$.
\end{rem}

Thanks to lemma \ref{l:f->Df}, we can equally deal with $\rho$ or $\nabla\rho$: as the equation for $\rho$ is very simple, we
choose to work with it. Keeping in mind
that $\left[X(t)\,,\,\d_t+u\cdot\nabla\right]=0$, we have
$$
\d_t\left(\d_X\rho\right)\,+\,u\cdot\nabla\left(\d_X\rho\right)\,=\,0\,,
$$
from which (remember also \eqref{est:Df->f}) it immediately follows that
\begin{equation} \label{est:d_X-rho}
 \left\|\d_{X(t)}\rho(t)\right\|_{\mc{C}^\veps}\,\leq\,
C\left(\,\wtilde{\|}X_0\|_{\mc{C}^\veps}\,\|\rho_0\|_{\mc{C}^1_*}\,+\,
\left\|\d_{X_0}\nabla\rho_0\right\|_{\mc{C}^{\veps-1}}\right)\exp\left(c\int^t_0\|\nabla u\|_{L^\infty}\,d\tau\right)\,.
\end{equation}
Therefore, keeping in mind also \eqref{est:X_C^e}, one gets also
\begin{eqnarray} \label{est:d_X-Drho}
 \left\|\d_{X(t)}\nabla\rho(t)\right\|_{\mc{C}^{\veps-1}} & \leq & C\,\exp\left(\int^t_0\|\nabla u\|_{L^\infty}\,d\tau\right)\times \\
& & \times\left(\|\rho_0\|_{\mc{C}^1_*}\,\wtilde{\|}X_0\|_{\mc{C}^\veps}\,+\,
\left\|\d_{X_0}\nabla\rho_0\right\|_{\mc{C}^{\veps-1}}\,+\right. \nonumber \\
& & \qquad\qquad\qquad\qquad\qquad\left.+\,
\int^t_0e^{-\int^\tau_0\|\nabla u\|_{L^\infty}d\tau'}\left\|\d_X\Omega\right\|_{\mc{C}^{\veps-1}}\,d\tau\right). \nonumber
\end{eqnarray}

\subsubsection{Striated regularity for the pressure term}

In this paragraph we want to show that geometric properties propagates also to the pressure term,
i.e. we want to prove $\d_X\nabla\Pi\in\mc{C}^\veps$.

Again, we use the decomposition $\;\d_X\nabla\Pi\,=\,T_X\left(\nabla\Pi\right)\,+\,\left(\d_X-T_X\right)\nabla\Pi$.

As usual, lemma \ref{l:T_X} gives us
$$
\left\|\left(\d_X\,-\,T_X\right)\nabla\Pi\right\|_{\mc{C}^\veps}\,\leq\,
C\,\wtilde{\|}X\|_{\mc{C}^\veps}\,\left\|\nabla^2\Pi\right\|_{L^\infty}\,.
$$
Now we use estimate \eqref{est:D^2-Pi}, the fact that $\log(e+\zeta)\leq e+\zeta^{1/2}$ and Young's inequality
to isolate the term $\left\|\d_X\nabla\Pi\right\|_{\mc{C}^\veps}$. As $2z\leq1+z^2$, we have
$$
\wtilde{\|}X\|^2_{\mc{C}^\veps}\,\|\nabla\Pi\|_{\mc{C}^1_*}\,\leq\,C\left(
\wtilde{\|}X\|_{\mc{C}^\veps}\,\|\nabla\Pi\|_{\mc{C}^1_*}\,+\,\wtilde{\|}X\|^3_{\mc{C}^\veps}\,\|\nabla\Pi\|_{\mc{C}^1_*}\right)\,,
$$
and finally we can control $\left\|\left(\d_X\,-\,T_X\right)\nabla\Pi\right\|_{\mc{C}^\veps}$ by the quantity
\begin{equation} \label{est:d_X-T_X-Pi}
C\left(\|\rho\|_{W^{1,\infty}}\,\wtilde{\|}X\|_{\mc{C}^\veps}\,\|\nabla\Pi\|_{\mc{C}^1_*}\,+\,
\wtilde{\|}X\|_{\mc{C}^\veps}\,\left\|\nabla u\right\|^2_{L^\infty}\,+\,
\wtilde{\|}X\|^3_{\mc{C}^\veps}\,\|\nabla\Pi\|_{\mc{C}^1_*}\right)+\,
\frac{1}{2}\,\left\|\d_X\nabla\Pi\right\|_{\mc{C}^\veps}.
\end{equation}

To deal with the paravector term, we keep in mind that $\nabla\Pi\,=\,\nabla\left(-\Delta\right)^{-1}\left(g_1\,+\,g_2\right)$,
where we have set
$$
g_1\,=\,-\,\nabla\left(\log\rho\right)\cdot\nabla\Pi \qquad\mbox{ and }\qquad
g_2\,=\,\rho\,\,\div(u\cdot\nabla u)\,.
$$
So it's enough to prove that both $T_X\nabla\left(-\Delta\right)^{-1}g_1$ and
$T_X\nabla\left(-\Delta\right)^{-1}g_2$ belong to $\mc{C}^\veps$.

Let us consider first the term
\begin{equation} \label{eq:T_X-Op}
T_X\nabla\left(-\Delta\right)^{-1}g_2\,=\,\nabla\left(-\Delta\right)^{-1}T_Xg_2\,+\,
\left[T_X,\nabla\left(-\Delta\right)^{-1}\right]g_2\,.
\end{equation}
From lemma \ref{l:comm} one immediately gets that
\begin{equation} \label{est:g_2-comm}
\left\|\left[T_X,\nabla\left(-\Delta\right)^{-1}\right]g_2\right\|_{\mc{C}^\veps}\,\leq\,C\,\wtilde{\|}X\|_{\mc{C}^\veps}\,
\|g_2\|_{\mc{C}^0_*}\,\leq\,C\,\rho^*\,\wtilde{\|}X\|_{\mc{C}^\veps}\,\|\nabla u\|^2_{L^\infty}\,,
\end{equation}
while it's obvious that
$$
\left\|\nabla\left(-\Delta\right)^{-1}T_Xg_2\right\|_{\mc{C}^\veps}\,\leq\,C\,\left\|T_Xg_2\right\|_{\mc{C}^{\veps-1}}\,.
$$
Now we use Bony's paraproduct decomposition and write
$$
T_Xg_2\,=\,T_XT_\rho\div(u\cdot\nabla u)\,+\,T_XT_{\div(u\cdot\nabla u)}\rho\,+\,T_XR\left(\rho,\div(u\cdot\nabla u)\right)\,.
$$
From proposition \ref{p:op} and the equality $\div(u\cdot\nabla u)=\nabla u:\nabla u$, it follows that
\begin{equation} \label{est:g_2-pp1}
\left\|T_XT_{\div(u\cdot\nabla u)}\rho\right\|_{\mc{C}^{\veps-1}}\,\leq\,C\,\|X\|_{L^\infty}\,
\left\|T_{\div(u\cdot\nabla u)}\rho\right\|_{\mc{C}^\veps}\,\leq\,C\,\wtilde{\|}X\|_{\mc{C}^\veps}\,\|\rho\|_{\mc{C}^1_*}\,
\|\nabla u\|^2_{L^\infty}\,,
\end{equation}
and the same estimate holds true for the remainder term $T_XR\left(\rho,\div(u\cdot\nabla u)\right)$. Lemma \ref{l:pvec-pprod},
instead, provides a control for $\left\|T_XT_\rho\div(u\cdot\nabla u)\right\|_{\mc{C}^{\veps-1}}$ by
(up to multiplication by a constant)
$$
\|X\|_{\mc{C}^\veps}\,\|\rho\|_{\mc{C}^1_*}\,\|\nabla u\|^2_{L^\infty}\,+\,\|\nabla u\|^2_{L^\infty}\,
\|T_X\rho\|_{\mc{C}^{\veps-1}}\,+\,
\|\rho\|_{\mc{C}^1_*}\,\|T_X\div(u\cdot\nabla u)\|_{\mc{C}^{\veps-1}}\,,
$$
where $\|T_X\rho\|_{\mc{C}^{\veps-1}}\leq C \|X\|_{\mc{C}^\veps}\,\|\rho\|_{\mc{C}^1_*}$ by proposition \ref{p:op}.
Now the problem is the control of the $\mc{C}^{\veps-1}$ norm of $T_X\div(u\cdot\nabla u)$. Writing
\begin{eqnarray*}
 T_X\div(u\cdot\nabla u) & = & \sum_{i,j}\,2\,T_XT_{\d_iu^j}\d_ju^i\,+\,T_X\d_iR(u^j,\d_ju^i) \\
& = & \sum_{i,j,k}\,2\,T_{X^k}\d_kT_{\d_iu^j}\d_ju^i\,+\d_iT_{X^k}\d_kR(u^j,\d_ju^i)\,-\,
T_{\d_iX^k}\d_kR(u^j,\d_ju^i)\,,
\end{eqnarray*}
by use of lemma \ref{l:pvec-pprod} we can easily see that it's bounded by
$$
\wtilde{\|}X\|_{\mc{C}^\veps}\,\|u\|_{\mc{C}^1_*}\,\|\nabla u\|_{L^\infty}\,+\,
\|u\|_{\mc{C}^1_*}\,\left\|T_X\nabla u\right\|_{\mc{C}^{\veps-1}}\,+\,
\|\nabla u\|_{L^\infty}\,\left\|T_Xu\right\|_{\mc{C}^\veps}\,.
$$
Hence, keeping in mind lemmas \ref{l:T_X} and \ref{l:f->Df}, we discover
$$
\left\|T_X\div(u\cdot\nabla u)\right\|_{\mc{C}^{\veps-1}}\,\leq\,C\left(
\,\wtilde{\|}X\|_{\mc{C}^\veps}\,\|u\|^2_{\mc{C}^1_*}\,+\,
\|\d_Xu\|_{\mc{C}^\veps}\,\|u\|_{\mc{C}^1_*}\right)\,,
$$
and therefore
\begin{equation} \label{est:g_2-pp2}
 \left\|T_XT_\rho\div(u\cdot\nabla u)\right\|_{\mc{C}^{\veps-1}}\,\leq\,C\left(\wtilde{\|}X\|_{\mc{C}^\veps}\,
\|\rho\|_{\mc{C}^1_*}\,\|u\|^2_{\mc{C}^1_*}\,+\,\|\rho\|_{\mc{C}^1_*}\,\|u\|_{\mc{C}^1_*}\,\|\d_Xu\|_{\mc{C}^\veps}\right).
\end{equation}
Putting inequalities \eqref{est:g_2-comm}, \eqref{est:g_2-pp1} and \eqref{est:g_2-pp2} all together, we finally get
\begin{equation} \label{est:g_2}
 \left\|T_X\nabla\left(-\Delta\right)^{-1}g_2\right\|_{\mc{C}^\veps}\,\leq\,C\left(\wtilde{\|}X\|_{\mc{C}^\veps}\,
\|\rho\|_{\mc{C}^1_*}\,\|u\|^2_{\mc{C}^1_*}\,+\,\|\rho\|_{\mc{C}^1_*}\,\|u\|_{\mc{C}^1_*}\,\|\d_Xu\|_{\mc{C}^\veps}\right),
\end{equation}
for some constant $C$ which depends also on $\rho^*$ and $\rho_*$.

Before going on, let us state a simple lemma.
\begin{lemma} \label{l:d_X-F}
Fix a $\veps\in\,]0,1[$ and an open interval $I\subset\R$. \\
Let $X$ be a $\mc{C}^\veps$ vector-field with divergence in $\mc{C}^\veps$ and $F:I\ra\R$ be a smooth function.

Then, for all compact set $J\subset I$ and all $\rho\in W^{1,\infty}$ valued in $J$ and such that $\d_X\rho\in\mc{C}^\veps$, one has that
$\d_X(F\circ\rho)\,\in\,\mc{C}^\veps$ and $\d_X\nabla(F\circ\rho)\,\in\,\mc{C}^{\veps-1}$. Moreover, the following estimates hold:
\begin{eqnarray*}
 \left\|\d_X(F\circ\rho)\right\|_{\mc{C}^\veps} & \leq & C\,\|\rho\|_{W^{1,\infty}}\,\left\|\d_X\rho\right\|_{\mc{C}^\veps} \\
\left\|\d_X\nabla(F\circ\rho)\right\|_{\mc{C}^{\veps-1}} & \leq & C\,\|\rho\|_{W^{1,\infty}}\,
\left(\left\|\d_X\rho\right\|_{\mc{C}^\veps}\,+\,\wtilde{\|}X\|_{\mc{C}^\veps}\,\|\rho\|_{W^{1,\infty}}\right)\,,
\end{eqnarray*}
for a constant $C$ depending only on $F$ and on the fixed subset $J$.
\end{lemma}
 \begin{proof}
  The first inequality is immediate keeping in mind the identity $\d_X(F\circ\rho)\,=\,F'(\rho)\,\d_X\rho$ and the estimate
$$
\left\|F'(\rho)\right\|_{\mc{C}^\veps}\,\leq\,C\,\left\|F''\right\|_{L^\infty(J)}\,\|\rho\|_{\mc{C}^\veps}\,\leq\,
C\,\left\|F''\right\|_{L^\infty(J)}\,\|\rho\|_{W^{1,\infty}}\,.
$$
For the second one, we write:
$$
\d_X\nabla(F\circ\rho)\,=\,\d_X\left(F'(\rho)\,\nabla\rho\right)\,=\,F'(\rho)\,\d_X\nabla\rho\,+\,F''(\rho)\,\,\d_X\rho\,\nabla\rho\,.
$$
Let us observe that the first term is well-defined in $\mc{C}^{\veps-1}$, and using decomposition in paraproducts and remainder
operators, we have
$$
\left\|F'(\rho)\,\d_X\nabla\rho\right\|_{\mc{C}^{\veps-1}}\,\leq\,C\,\left\|F'(\rho)\right\|_{W^{1,\infty}}\,
\left\|\d_X\nabla\rho\right\|_{\mc{C}^{\veps-1}}\,.
$$
Now, the thesis immediately follows from lemma \ref{l:f->Df}.
 \end{proof}

Let us come back to $g_1$: using the same trick as in \eqref{eq:T_X-Op}, it's enough to estimate
$$
\left\|T_Xg_1\right\|_{\mc{C}^{\veps-1}}\qquad\mbox{ and }\qquad
\left\|\left[T_X,\nabla\left(-\Delta\right)^{-1}\right]g_1\right\|_{\mc{C}^\veps}\,.
$$
Again, the control of the commutator term follows from lemma \ref{l:comm}:
\begin{equation} \label{est:g_1-comm}
\left\|\left[T_X,\nabla\left(-\Delta\right)^{-1}\right]g_1\right\|_{\mc{C}^\veps}\,\leq\,
C\,\wtilde{\|}X\|_{\mc{C}^\veps}\,\|g_1\|_{\mc{C}^0_*}\,\leq\,
\frac{C}{\rho_*}\,\wtilde{\|}X\|_{\mc{C}^\veps}\,\|\nabla\rho\|_{L^\infty}\,\|\nabla\Pi\|_{\mc{C}^1_*}\,.
\end{equation}
For the other term, we use again Bony's paraproduct decomposition:
$$
T_Xg_1\,=\,T_XT_{\nabla(\log\rho)}\nabla\Pi\,+\,T_XT_{\nabla\Pi}\nabla(\log\rho)\,+\,T_XR(\nabla(\log\rho),\nabla\Pi)\,.
$$
Thanks to proposition \ref{p:op} we immediately find
\begin{equation} \label{est:g_1-pp1}
\left\|T_XT_{\nabla(\log\rho)}\nabla\Pi\right\|_{\mc{C}^{\veps-1}}\,\leq\,C\,\wtilde{\|}X\|_{\mc{C}^\veps}\,
\|\nabla\rho\|_{L^\infty}\,\|\nabla\Pi\|_{\mc{C}^1_*}\,,
\end{equation}
and the same control holds true also for the remainder. Moreover, a direct application of lemma \ref{l:pvec-pprod} implies
\begin{equation} \label{est:g_1-pp2}
\left\|T_XT_{\nabla(\log\rho)}\nabla\Pi\right\|_{\mc{C}^{\veps-1}}\,\leq\,C\left(\wtilde{\|}X\|_{\mc{C}^\veps}\,
\|\nabla\rho\|_{L^\infty}\,\|\nabla\Pi\|_{\mc{C}^1_*}\,+\,
\|\nabla\Pi\|_{\mc{C}^1_*}\,\left\|T_X\nabla(\log\rho)\right\|_{\mc{C}^{\veps-1}}\right),
\end{equation}
Now, from lemmas \ref{l:T_X} and \ref{l:d_X-F} we easily get
$$
\|T_X\nabla(\log\rho)\|_{\mc{C}^{\veps-1}}\,\leq\,C\left(\|\d_X\rho\|_{\mc{C}^\veps}\,\|\rho\|_{W^{1,\infty}}\,+\,
\wtilde{\|}X\|_{\mc{C}^\veps}\,\|\rho\|^2_{W^{1,\infty}}\right)\,.
$$
Putting this last relation into \eqref{est:g_1-pp2} and keeping in mind inequalities \eqref{est:g_1-comm} and \eqref{est:g_1-pp1},
we find
\begin{equation} \label{est:g_1}
 \left\|T_X\nabla\left(-\Delta\right)^{-1}g_1\right\|_{\mc{C}^\veps}\,\leq\,C\left(
\|\d_X\rho\|_{\mc{C}^\veps}\,\|\rho\|_{W^{1,\infty}}\,\|\nabla\Pi\|_{\mc{C}^1_*}\,+\,
\wtilde{\|}X\|_{\mc{C}^\veps}\,\|\rho\|^2_{W^{1,\infty}}\,\|\nabla\Pi\|_{\mc{C}^1_*}\right),
\end{equation}
where, as before, $C$ may depend also on $\rho^*$ and $\rho_*$.

Therefore, putting \eqref{est:d_X-T_X-Pi}, \eqref{est:g_2} and \eqref{est:g_1} together, we finally get
\begin{eqnarray}
 \left\|\d_X\nabla\Pi\right\|_{\mc{C}^\veps} & \leq & C\biggl(\|\rho\|_{W^{1,\infty}}
\left\|\d_X\rho\right\|_{\mc{C}^\veps}\|\nabla\Pi\|_{\mc{C}^1_*}\,+\,
\|\nabla\Pi\|_{\mc{C}^1_*}\wtilde{\|}X\|_{\mc{C}^\veps}\|\rho\|^2_{W^{1,\infty}}\,+ \label{est:d_X-Pi} \\
& & \quad\quad+\,\wtilde{\|}X\|^3_{\mc{C}^\veps}\|\nabla\Pi\|_{\mc{C}^1_*}\,+\,
\|\rho\|_{W^{1,\infty}}\wtilde{\|}X\|_{\mc{C}^\veps}\|u\|^2_{\mc{C}^1_*}\,+\,
\|\rho\|_{\mc{C}^1_*}\|u\|_{\mc{C}^1_*}\left\|\d_Xu\right\|_{\mc{C}^\veps}\biggr). \nonumber
\end{eqnarray}

\subsubsection{Conservation of striated regularity for the vorticity}

Let us now establish a control on $\|\d_X\Omega\|_{\mc{C}^{\veps-1}}$.
Applying the operator $\d_X$ to \eqref{eq:vort}, we obtain the evolution equation for $\d_X\Omega$:
\begin{equation} \label{eq:vort_str}
 \d_t\left(\d_X\Omega\right)\,+\,u\cdot\nabla\left(\d_X\Omega\right)\,=\,
\d_X\left(\frac{1}{\rho^2}\,\nabla\rho\wedge\nabla\Pi\right)\,-\,
\d_X\left(\Omega\cdot\nabla u\right)\,-\,\d_X\left(^t\nabla u\cdot\Omega\right)\,.
\end{equation}

The second and third terms of the right-hand side of \eqref{eq:vort_str} can be treated using once again the decomposition
$$
\d_X\left(\Omega\cdot\nabla u\,+\,^t\nabla u\cdot\Omega\right)\,=\,
\left(\d_X\,-\,T_X\right)\left(\Omega\cdot\nabla u\,+\,^t\nabla u\cdot\Omega\right)\,+\,
T_X\left(\Omega\cdot\nabla u\,+\,^t\nabla u\cdot\Omega\right)\,.
$$
Lemma \ref{l:T_X} says that the operator $\d_X-T_X$ maps $\mc{C}^0_*$ in $\mc{C}^{\veps-1}$ continuously; as
$L^\infty\hookrightarrow\mc{C}^0_*$, one has
$$
 \left\|\left(\d_X\,-\,T_X\right)\left(\Omega\cdot\nabla u\,+\,^t\nabla u\cdot\Omega\right)\right\|_{\mc{C}^{\veps-1}}\,\leq\,C\,
\wtilde{\|}X\|_{\mc{C}^\veps}\,\|\Omega\|_{L^\infty}\,\|\nabla u\|_{L^\infty}\,.
$$
To handle the paravector term, we proceed in the following way. First of all, we note that, as $\div u=0$, we can write
$$
\left(\Omega\cdot\nabla u\,+\,^t\nabla u\cdot\Omega\right)_{ij}\,=\,\sum_k
\left(\d_iu^k\,\d_ku^j\,-\,\d_ju^k\,\d_ku^i\right)\,=\,
\sum_k\left(\d_k\left(u^j\,\d_iu^k\right)\,-\,\d_k\left(u^i\,\d_ju^k\right)\right)\,.
$$
So, we have to estimate the $\mc{C}^{\veps-1}$ norm of terms of the type $\,T_XT_{\nabla u}\nabla u\,$ and
$\,T_X\nabla R(u,\nabla u)\,$. For the former we apply directly lemma \ref{l:pvec-pprod}, while for the latter
we first use the same trick as in \eqref{eq:T_X-Op} and then lemmas
\ref{l:pvec-pprod} and \ref{l:comm}:
\begin{eqnarray*}
 \left\|T_XT_{\nabla u}\nabla u\right\|_{\mc{C}^{\veps-1}} & \leq & C\left(
\wtilde{\|}X\|_{\mc{C}^\veps}\,\|\nabla u\|^2_{L^\infty}\,+\,\|\nabla u\|_{L^\infty}\,\|T_X\nabla u\|_{\mc{C}^{\veps-1}}\right)\\
\left\|T_X\nabla R(u,\nabla u)\right\|_{\mc{C}^{\veps-1}} & \leq &
C\left(\wtilde{\|}X\|_{\mc{C}^\veps}\|u\|_{\mc{C}^1_*}\|\nabla u\|_{L^\infty}+
\|u\|_{\mc{C}^1_*}\|T_X\nabla u\|_{\mc{C}^{\veps-1}}+\|\nabla u\|_{L^\infty}\|T_Xu\|_{\mc{C}^\veps}\right).
\end{eqnarray*}
Hence, from lemmas \ref{l:T_X} and \ref{l:f->Df} it easily follows
\begin{equation} \label{est:Om-Du}
 \left\|\d_X\left(\Omega\cdot\nabla u\,+\,^t\nabla u\cdot\Omega\right)\right\|_{\mc{C}^{\veps-1}}\,\leq\,C\left(\,
\wtilde{\|}X\|_{\mc{C}^\veps}\,\|u\|^2_{\mc{L}^{p,\infty}}\,+\,
\left\|\d_Xu\right\|_{\mc{C}^\veps}\,\|u\|_{\mc{L}^{p,\infty}}\right).
\end{equation}

Now, let us analyse the first term of \eqref{eq:vort_str}. It can be written as the sum of three items:
$$
\d_X\left(\frac{1}{\rho^2}\,\nabla\rho\wedge\nabla\Pi\right)\,=\,-\,
\frac{2}{\rho^3}\left(\d_X\rho\right)\left(\nabla\rho\wedge\nabla\Pi\right)\,+\,
\frac{1}{\rho^2}\,\left(\d_X\nabla\rho\right)\wedge\nabla\Pi\,+\,\frac{1}{\rho^2}\,\nabla\rho\wedge\left(\d_X\nabla\Pi\right)\,.
$$
So, let us consider each one separately and prove that it belongs to the space $\mc{C}^{\veps-1}$.

First and third terms are in fact in $L^\infty\hookrightarrow\mc{C}^{\veps-1}$, satisfy
\begin{eqnarray*}
\left\|\frac{1}{\rho^3}\left(\d_X\rho\right)\left(\nabla\rho\wedge\nabla\Pi\right)\right\|_{\mc{C}^{\veps-1}} & \leq &  
\frac{C}{\left(\rho_*\right)^3}\,\|\d_X\rho\|_{\mc{C}^\veps}\,\|\nabla\rho\|_{L^\infty}\,\|\nabla\Pi\|_{\mc{C}^1_*} \\
\left\|\frac{1}{\rho^2}\,\nabla\rho\wedge\left(\d_X\nabla\Pi\right)\right\|_{\mc{C}^{\veps-1}} & \leq &
\frac{C}{\left(\rho_*\right)^2}\,\left\|\nabla\rho\right\|_{L^\infty}\,\left\|\d_X\nabla\Pi\right\|_{\mc{C^\veps}}\,.
\end{eqnarray*}

Now, let us find a $\mc{C}^{\veps-1}$ control for the second term. Note that it is well-defined,
due to the fact that both $\rho$ and $\nabla\Pi$ are in $\mc{C}^1_*$ (the product of a $\mc{C}^{\sigma}$
function, $\sigma<0$, with a $L^\infty$ one is not even well-defined).
With a little abuse of notation (in the end, we have to deal with the sum of
products of components of the two vector-fields), we write
$$
\left(\d_X\nabla\rho\right)\nabla\Pi\,=\,T_{\left(\d_X\nabla\rho\right)}\nabla\Pi\,+\,T_{\nabla\Pi}\left(\d_X\nabla\rho\right)\,+\,
R\left(\d_X\nabla\rho,\nabla\Pi\right)\,;
$$
remembering proposition \ref{p:op} and the embeddings
$\mc{C}^1_*\hookrightarrow L^\infty\hookrightarrow\mc{C}^0_*$, we get
$$
\left\|\left(\d_X\nabla\rho\right)\wedge\nabla\Pi\right\|_{\mc{C}^{\veps-1}}\,\leq\,C\,\left\|\d_X\nabla\rho\right\|_{\mc{C}^{\veps-1}}\,
\left\|\nabla\Pi\right\|_{\mc{C}^1_*}\,.
$$
In the same way, as $\left\|1/\rho^2\right\|_{\mc{C}^1_*}\leq\left\|1/\rho^2\right\|_{W^{1,\infty}}$, we get
$$
\left\|\frac{1}{\rho^2}\left(\d_X\nabla\rho\right)\wedge\nabla\Pi\right\|_{\mc{C}^{\veps-1}}\,\leq\,
\frac{C}{\left(\rho_*\right)^2}
\left(1+\frac{\|\nabla\rho\|_{L^\infty}}{\rho_*}\right)\,\left\|\d_X\nabla\rho\right\|_{\mc{C}^{\veps-1}}\,
\left\|\nabla\Pi\right\|_{\mc{C}^1_*}\,.
$$

So, using also lemma \ref{l:f->Df}, we finally obtain, for a constant $C$ depending also on $\rho_*$ and $\rho^*$,
\begin{eqnarray}
 \left\|\d_X\left(\frac{1}{\rho^2}\,\nabla\rho\wedge\nabla\Pi\right)\right\|_{\mc{C}^{\veps-1}} & \leq &
C\left(\wtilde{\|}X\|_{\mc{C}^\veps}\,\|\rho\|^2_{W^{1,\infty}}\,\|\nabla\Pi\|_{\mc{C}^1_*}\,+\,
\left\|\nabla\rho\right\|_{L^\infty}\,\left\|\d_X\nabla\Pi\right\|_{\mc{C^\veps}}\,+\right. \label{est:wedge} \\
& & \;\;\quad\qquad\qquad\qquad\qquad\left.+\,
\|\rho\|_{W^{1,\infty}}\,\left\|\d_X\nabla\rho\right\|_{\mc{C}^{\veps-1}}\,
\left\|\nabla\Pi\right\|_{\mc{C}^1_*}\right). \nonumber
\end{eqnarray}

Therefore, from equation \eqref{eq:vort_str}, classical estimates for transport equation in H\"older spaces
and inequalities \eqref{est:Om-Du} and \eqref{est:wedge}, we obtain
\begin{eqnarray}
 \left\|\d_X\Omega(t)\right\|_{\mc{C}^{\veps-1}} & \leq & C\,\exp\left(c\int^t_0\|\nabla u\|_{L^\infty}\,d\tau\right)
\left(\left\|\d_{X_0}\Omega_0\right\|_{\mc{C}^{\veps-1}}\,+\,\int^t_0e^{-\int^t_0\|\nabla u\|_{L^\infty}d\tau'}\times\right.
\label{est:d_X-Om} \\
& & \quad\times\left(\,\,\wtilde{\|}X\|_{\mc{C}^\veps}\|u\|^2_{\mc{L}^{p,\infty}}\,+\,
\|\d_Xu\|_{\mc{C}^\veps}\|u\|_{\mc{L}^{p,\infty}}\,+\right. \nonumber \\
& & \quad\qquad+\,\wtilde{\|}X\|_{\mc{C}^\veps}\,\|\rho\|^2_{W^{1,\infty}}\,
\left\|\nabla\Pi\right\|_{\mc{C}^1_*}\,+\,\|\nabla\rho\|_{L^\infty}\,\left\|\d_X\nabla\Pi\right\|_{\mc{C}^\veps}\,+ \nonumber \\
& & \left.\quad\qquad\left.+\,\|\rho\|_{W^{1,\infty}}\,\left\|\nabla\Pi\right\|_{\mc{C}^1_*}\,
\left\|\d_X\nabla\rho\right\|_{\mc{C}^{\veps-1}}\right)\,d\tau\right)\,. \nonumber
\end{eqnarray}

\subsection{Final estimates} \label{ss:fin-est}

First of all, thanks to Young's inequality and estimates \eqref{est:Pi_L^2} and \eqref{est:Pi_C^1_*}, for all $\eta\in\,[0,1]$
we have
\begin{equation} \label{f-est:Pi}
\|\nabla\Pi\|^\eta_{L^2}\,\|\nabla\Pi\|^{1-\eta}_{L^\infty}\,\leq\,
\left\|\nabla\Pi\right\|_{L^2\cap\,\mc{C}^1_*}\,\leq\,C\left(
\left(1+\|\nabla\rho\|^\delta_{L^\infty}\right)\|u\|_{L^p}\,\|\Omega\|_{L^q\cap L^\infty}\,+\,\rho^*\,\|\nabla u\|^2_{L^\infty}\right).
\end{equation}
So, setting
$$
L(t)\,:=\,\|u(t)\|_{L^p}\,+\,\|\Omega(t)\|_{L^q\cap L^\infty}\,,
$$
putting \eqref{est:Drho_L^inf} and \eqref{f-est:Pi} into \eqref{est:u_L^p}, \eqref{est:Om_L^q} and \eqref{est:Om_L^inf}, for all fixed
$T>0$ we obtain, in the time interval $[0,T]$, an inequality of the form
$$
L(t)\,\leq\,C\,\exp\left(c\int^t_0\|\nabla u\|_{L^\infty}d\tau\right)\left(L(0)\,+\,
\int^t_0\left\|\nabla u\right\|^2_{L^\infty}\,d\tau\,+\,\int^t_0L^2(\tau)\,d\tau\right)\,,
$$
with constants $C$, $c$ depending only on $N$, $\veps$, $\rho_*$ and $\rho^*$. Now, if we define
\begin{equation} \label{cond-T_1}
T\,:=\,\sup\left\{t>0\,\biggl|\,\int^t_0\left(e^{-\int^\tau_0L(\tau')d\tau'}
L(\tau)+\left\|\nabla u(\tau)\right\|^2_{L^\infty}\right)d\tau\,\leq\,2L(0)\right\}\,,
\end{equation}
from previous inequality and Gronwall's lemma and applying a standard bootstrap procedure,
we manage to estimate the norms of the solution on $[0,T]$ in terms of initial data only:
$$
L(t)\,\leq\,C\,L(0)\qquad\mbox{ and }\qquad \|\rho(t)\|_{W^{1,\infty}}\,\leq\,C\,\|\rho_0\|_{W^{1,\infty}}\,.
$$
From this, keeping in mind \eqref{f-est:Pi} and \eqref{cond-T_1}, we also have
$$
 \|\nabla\Pi\|_{L^\infty_t(L^2)\cap L^1_t(\mc{C}^1_*)}\,\leq\,C\left(1\,+\,\|\nabla\rho_0\|^\delta_{L^\infty}\right)L^2(0)\,.
$$

Now, let us focus on estimates about striated regularity.
First of all, from \eqref{est:I} we get that the family $X(t)$ remains non-degenerate: $I(X(t))\,\geq\,C\,I(X_0)$.

Now, for notation convenience, let us come back to the case of only one vector-field, which we keep to call $X$, and set
$S(t):=\left\|\d_{X(t)}\Omega(t)\right\|_{\mc{C}^{\veps-1}}$. Let us note that the constants $C$ which will
occur in our estimates depend on the functional norms of the initial data, but also on the time $T$.

From \eqref{est:X_C^e} and \eqref{est:d_X-u} we find
\begin{eqnarray*}
 \wtilde{\|}X(t)\|_{\mc{C}^\veps} & \leq & C\left(\wtilde{\|}X_0\|_{\mc{C}^\veps}\,+\,\int^t_0S(\tau)\,d\tau\right) \\
\left\|\d_{X(t)}u(t)\right\|_{\mc{C}^\veps} & \leq & C\left(S(t)\,+\,\wtilde{\|}X(t)\|_{\mc{C}^\veps}\,
\|\nabla u(t)\|_{L^\infty}\right),
\end{eqnarray*}
while \eqref{est:d_X-rho} and \eqref{est:d_X-Drho} give us
$$
\left\|\d_X\rho\right\|_{\mc{C^\veps}}\,\leq\,C
\quad\mbox{ and }\quad
\left\|\d_{X(t)}\nabla\rho(t)\right\|_{\mc{C}^{\veps-1}}\,\leq\,C\left(1\,+\,\int^t_0S(\tau)\,d\tau\right).
$$

Before going on, let us notice the following fact, which is a direct consequence of the integral condition
in \eqref{cond-T_1}: for $m=1$, $2$ we have
\begin{equation} \label{f-est:S-u}
\int^t_0\left(\int^\tau_0S(\tau')d\tau'\right)\|\nabla u(\tau)\|^m_{L^\infty}\,d\tau\,\leq\,C\,\int^t_0S(\tau)\,d\tau\,.
\end{equation}
We will repeatedly use it in what follows.

Now, let us focus on $\d_X\nabla\Pi$: for convenience, we want to estimate its $L^1_t(\mc{C}^\veps)$ norm,
starting from the bound \eqref{est:d_X-Pi} and the ones we have just found. \\
First of all, we have
\begin{eqnarray*}
\int^t_0\|\rho\|_{W^{1,\infty}}\left\|\d_X\rho\right\|_{\mc{C}^\veps}\|\nabla\Pi\|_{\mc{C}^1_*}\,d\tau & \leq & C \\
\int^t_0\|\nabla\Pi\|_{\mc{C}^1_*}\wtilde{\|}X\|_{\mc{C}^\veps}\|\rho\|^2_{W^{1,\infty}}\,d\tau & \leq & 
\|\nabla\Pi\|_{L^1_t(\mc{C}^1_*)}\wtilde{\|}X\|_{L^\infty_t(\mc{C}^\veps)}\;\leq\;C\left(1+\int^t_0S(\tau)d\tau\right).
\end{eqnarray*}
Exactly in the same way, using also Jensen's inequality, we get
$$
\int^t_0\wtilde{\|}X\|^3_{\mc{C}^\veps}\|\nabla\Pi\|_{\mc{C}^1_*}\,d\tau\,\leq\,
C\|\nabla\Pi\|_{L^1_t(\mc{C}^1_*)}\wtilde{\|}X\|^3_{L^\infty_t(\mc{C}^\veps)}\,\leq\,
C\left(1+\int^t_0S^3(\tau)\,d\tau\right)\,,
$$
while, keeping in mind the definition of the $\mc{L}^{p,\infty}$ norm (see remark \ref{r:lem_f->Df}) and
inequality \eqref{f-est:S-u}, we easily find
\begin{eqnarray*}
 \int^t_0\|\rho\|_{W^{1,\infty}}\wtilde{\|}X\|_{\mc{C}^\veps}\|u\|^2_{\mc{L}^{p,\infty}}\,d\tau & \leq & 
C\left(1+\int^t_0S(\tau)\,d\tau\right) \\
\int^t_0\|\rho\|_{W^{1,\infty}}\|u\|_{\mc{L}^{p,\infty}}\left\|\d_Xu\right\|_{\mc{C}^\veps}d\tau & \leq &
C\left(1+\int^t_0S(\tau)\,d\tau+\int^t_0\|\nabla u\|_{L^\infty}\,S(\tau)\,d\tau\right).
\end{eqnarray*}
Therefore, in the end we get
$$
\left\|\d_X\nabla\Pi\right\|_{L^1_t(\mc{C}^\veps)}\,\leq\,C\left(1\,+\,\int^t_0(1+\|\nabla u\|_{L^\infty})S(\tau)\,d\tau\,+
\,\int^t_0S^3(\tau)\,d\tau\right).
$$

Finally, let us handle the term $S(t)$: from \eqref{est:d_X-Om}, we see that we have to control
the $L^1_t$ norm of $\iota$, defined by \eqref{def:iota}.
First of all, we have
$$
\int^t_0\wtilde{\|}X\|_{\mc{C}^\veps}\|u\|^2_{\mc{L}^{p,\infty}}\,d\tau\,,\quad
\int^t_0\|u\|_{\mc{L}^{p,\infty}}\left\|\d_Xu\right\|_{\mc{C}^\veps}\,d\tau\;\leq\;C\left(1+\int^t_0S(\tau)\,d\tau\right)\,:
$$
we have just analysed the same items multiplied by $\|\rho\|_{W^{1,\infty}}$, which we controlled by a constant.
Moreover, one immediately find
$$
\int^t_0\|\nabla\rho\|_{L^\infty}\left\|\d_X\nabla\Pi\right\|_{\mc{C}^\veps}\,d\tau\,\leq\,
C\,\left\|\d_X\nabla\Pi\right\|_{L^1_t(\mc{C}^\veps)}\,,
$$
while the term $\wtilde{\|}X\|_{\mc{C}^\veps}\|\rho\|^2_{W^{1,\infty}}\left\|\nabla\Pi\right\|_{\mc{C}^1_*}$ already occurred
in considering $\d_X\nabla\Pi$, and so it can be absorbed in the previous inequality.
Finally, we have
$$
\int^t_0\|\rho\|_{W^{1,\infty}}\left\|\nabla\Pi\right\|_{\mc{C}^1_*}
\left\|\d_X\nabla\rho\right\|_{\mc{C}^{\veps-1}}\,d\tau\,\leq\,C\left(1+\int^t_0S(\tau)\,d\tau\right).
$$

Putting all these inequalitites together, in the end we find the control for $S(t)$ on $[0,T]$:
$$
S(t)\,\leq\,C\left(S(0)\,+\,\int^t_0(1+\|\nabla u\|_{L^\infty})S(\tau)d\tau\,+\,\int^t_0S^3(\tau)d\tau\right).
$$
Now, suppose that $T$ was chosen so small that, in addition to \eqref{cond-T_1}, for all $t\in[0,T]$ one has also
\begin{equation} \label{cond-T_2}
 \int^t_0S^3(\tau)\,d\tau\,\leq\,2\,S(0)\,.
\end{equation}
Then Gronwall's lemma allows us to get the bound
$$
\left\|\d_{X(t)}\Omega(t)\right\|_{\mc{C}^{\veps-1}}\,\leq\,C\,S(0)\qquad\forall\,t\in[0,T]\,,
$$
for a constant $C$ depending
only on $T$, $N$, $p$, $q$, $\veps$, $\rho_*$ and $\rho^*$ and on the norms of initial data in the relative functional spaces.

Let us note that this inequality allows us to recover a uniform bound, on $[0,T]$, for $\|\nabla u\|_{L^\infty}$ and
$\|\nabla\Pi\|_{\mc{C}^1_*}$, which we previously controlled only in $L^1_t$.

\begin{rem} \label{r:T}
The lifespan $T$ of the solution is essentially determined by conditions \eqref{cond-T_1} and \eqref{cond-T_2}. In
section \ref{s:lifespan} we will establish an explicit lower bound for $T$ in terms of the norms of initial data
only and we will compare it with the classical result in the case of constant density.
\end{rem}

\subsection{Proof of the existence of a solution}

After establishing a priori estimates, we want to give the proof of the existence of a solution for system \eqref{eq:ddeuler}
under our assumptions.

We will get it in a classical way: first of all, we will construct a sequence of approximate solutions of our problem,
for which a priori estimates of the previous section hold uniformly, and
then we will show the convergence of such a sequence to a solution of (\ref{eq:ddeuler}).

Now, we will work with positive times only, but it goes without saying that the same argument holds true also for the backward
evolution.

\subsubsection{Construction of a sequence of approximate solutions}

For each $n\in\mbb{N}$, let us define $u^n_0\,:=\,S_nu_0$; obviously $u^n_0\in L^p$, and an easy computation shows that it belongs
also to the space $B^\sigma_{p,r}$ for all $\sigma\in\mbb{R}$ and all $r\in[1,+\infty]$.
Let us notice that $\bigcap_{\sigma}\!B^\sigma_{p,r}\subset\mc{C}^\infty_b$, so
in particular we have that
$u^n_0\in L^p\cap B^s_{\infty,r}$, for some fixed $s>1$ and $r\in[1,+\infty]$ such that $B^s_{\infty,r}\hookrightarrow\mc{C}^{0,1}$.

Keeping in mind that $\left[S_n,\nabla\right]=0$, we have that $\Omega^n_0\,=\,S_n\Omega_0\in L^q\cap B^{s-1}_{\infty,r}$;
in particular, from \eqref{est:CZ} we get $\nabla u^n_0\in L^q$.

Now let us take an even radial function $\theta\in\mc{C}^\infty_0(\R^N)$, supported in the unitary ball, such that
$0\leq\theta\leq1$ and $\int_{\R^N}\theta(x)\,dx=1$, and set $\theta_n(x)=n^N\,\theta(nx)$ for all $n\in\N$. We
define $\rho^n_0\,:=\,\theta_n*\rho_0$: it belongs to $B^s_{\infty,r}$ and it satisfies the bounds $0<\rho_*\leq\rho^n_0\leq\rho^*$.

Moreover, by properties of localisation operators $S_n$ and of $\theta_n$, we have:
\begin{itemize}
 \item $\rho^n_0\rightharpoonup\rho_0$ in $W^{1,\infty}$ and
$\|\nabla\rho^n_0\|_{L^\infty}\,\leq\,c\,\|\nabla\rho_0\|_{L^\infty}$;
 \item $u^n_0\rightarrow u_0$ in the space $L^p$ and $\|u^n_0\|_{L^p}\,\leq\,c\,\|u_0\|_{L^p}$;
 \item $\Omega^n_0\rightarrow\Omega_0$ in $L^q$ and $\|\Omega^n_0\|_{L^q}\,\leq\,c\,\|\Omega_0\|_{L^q}$,
$\|\Omega^n_0\|_{L^\infty}\,\leq\,c\,\|\Omega_0\|_{L^\infty}$.
\end{itemize}

So, for each $n$, theorem 3 and remark 4 of \cite{D-F} give us a unique solution of \eqref{eq:ddeuler} such that:
\begin{itemize}
 \item [(i)] $\rho^n\in\mc{C}([0,T^n];B^s_{\infty,r})$, with $0<\rho_*\leq\rho^n\leq\rho^*$;
 \item [(ii)] $u^n\in\mc{C}([0,T^n];L^p\cap B^s_{\infty,r})$, with $\Omega^n\in\mc{C}([0,T^n];L^q\cap B^{s-1}_{\infty,r})$;
 \item [(iii)] $\nabla\Pi^n\in\mc{C}([0,T^n];L^2)\cap L^1([0,T^n];B^s_{\infty,r})$.
\end{itemize}

For such a solution, a priori estimates of the previous section hold at every step $n$. Moreover, remembering previous properties
about approximated initial data, we can find a control independent of $n\in\mbb{N}$.
So, we can find a positive time $T\leq T^n$ for all $n\in\N$, such that in $[0,T]$ the approximate solutions
are all defined and satisfy uniform bounds.

\subsubsection{Convergence of the sequence of approximate solutions}

To prove convergence of the obtained sequence, we appeal to a compactness argument. Actually, we weren't able to apply the classical
method used for the homogeneous case, i.e. proving estimates in rough spaces as $\mc{C}^{-\alpha}\,\,(\alpha>0)$:
we couldn't solve the elliptic equation for the pressure term in this framework.

\medskip
We know that $\left(\rho^n\right)_{n\in\mbb{N}}\,\subset L^\infty([0,T];W^{1,\infty})$,
$\left(u^n\right)_{n\in\mbb{N}}\,\subset L^\infty([0,T];L^p)$ and
$\left(\nabla\Pi^n\right)_{n\in\mbb{N}}\,\subset L^\infty([0,T];L^2)$ and, thanks to a priori estimates, all these sequences are
bounded in the respective functional spaces.

Due to the reflexivity of $L^2$ and $L^p$ and seeing $L^\infty$ as the dual of $L^1$, up to passing to a subsequence, we obtain
the existence of functions $\rho$, $u$ and $\nabla\Pi$ such that:
\begin{itemize}
 \item $\rho^n\,\stackrel{*}{\upra}\,\rho$ in the space $L^\infty([0,T];W^{1,\infty})$,
\item $u^n\,\upra\,u$ in $L^\infty([0,T];L^p)$ and
\item $\nabla\Pi^n\,\upra\,\nabla\Pi$ in $L^\infty([0,T];L^2)$.
\end{itemize}

Nevertheless, we are not able to prove that $\left(\rho,u,\nabla\Pi\right)$ is indeed a solution of system \eqref{eq:ddeuler}:
passing to the limit in nonlinear terms requires strong convergence in (even rough) suitable functional spaces. So let us
argue in a different way and establish strong convergence properties, which will be useful also to prove
preservation of striated regularity.

First of all, let us recall that, by construction, $u^n_0\,\ra\,u_0$ in $L^p$ and $\Omega^n_0\,\ra\,\Omega_0$ in $L^q$, and
$\left(\rho^n_0\right)_n$ is bounded in $W^{1,\infty}$. So, for $\alpha>0$ big enough (for instance, take
$\alpha=\max\left\{N/p\,,N/q\right\}$), we have that $\left(\rho^n_0\right)_n$, $\left(u^n_0\right)_n$, $\left(\Omega^n_0\right)_n$
are all bounded in the space $\mc{C}^{-\alpha}$.
\begin{rem} \label{r:conv-data}
 It goes without saying that the sequences of $u^n_0$ and $\Omega^n_0$ still converge in $\mc{C}^{-\alpha}$; moreover, also
$\rho^n_0\,\ra\,\rho_0$ in this space. Remember that $\rho_0$ belongs to the space $\mc{C}^1_*$, which coincides
(see \cite{C1995} for the proof) with the
Zygmund space, i.e. the set of bounded functions $f$ for which there exists a constant $Z_f$ such that
$$
\left|f(x+y)\,+\,f(x-y)\,-\,2\,f(x)\right|\,\leq\,Z_f\,\,|y|
$$
for all $x$, $y\,\in\R^N$. So, using the symmetry of $\theta$, we can write
$$
\rho^n_0(x)\,-\,\rho_0(x)\,=\,\frac{1}{2}\,n^N\int_{\R^N}\theta(ny)
\left(\rho_0(x+y)\,+\,\rho_0(x-y)\,-\,2\,\rho_0(x)\right)dy\,;
$$
from this identity we get that $\rho^n_0\,\ra\,\rho_0$ in $L^\infty$, and hence also in $\mc{C}^{-\alpha}$.
\end{rem}

Now, let us consider the equation for $\rho^n$:
$$
\d_t\rho^n\,=\,-\,u^n\cdot\nabla\rho^n\,.
$$
From a priori estimates we get that $\left(u^n\right)_n$ is bounded in $L^\infty([0,T];\mc{C}^1_*)$ and $\left(\nabla\rho^n\right)_n$
is bounded in the space $L^\infty([0,T];L^\infty)$; so, from proposition \ref{p:op}, one has that the
sequence
$\left(\d_t\rho^n\right)_n$ is bounded in $L^\infty([0,T];\mc{C}^{-\alpha})$. Therefore $\left(\rho^n\right)_n$ is bounded in
$\mc{C}^{0,1}([0,T];\mc{C}^{-\alpha})$, and in particular uniformly equicontinuous in the time variable.

Now, up to multiply by a $\vphi\in\mc{D}(\mbb{R}^N)$ (recall theorem 2.94 of \cite{B-C-D}) and extract a
subsequence, Ascoli-Arzel\`a theorem and Cantor diagonal process ensure us that $\rho^n\,\ra\,\rho$ in the space
$\mc{C}([0,T];\mc{C}^{-\alpha}_{loc})$.

Exactly in the same way, one can show that $\left(\rho^n\right)_n$ is bounded in $\mc{C}_b([0,T]\times\R^N)$ and it converges
to $\rho$ in this space.

Finally, remembering that $\rho\,\in\,L^\infty([0,T];W^{1,\infty})$ (recall the compactness argument), by interpolation we have
convergence also in $L^\infty([0,T];\mc{C}^{1-\eta}_{loc})$ for all $\eta>0$.

\medskip
We repeat the same argument for the velocity field. For all $n$, we have
$$
\d_tu^n\,=\,-\,u^n\cdot\nabla u^n\,-\,a^n\,\nabla\Pi^n\,,
$$
where we have set $a^n\,:=\,\left(\rho^n\right)^{-1}$. Let us notice that, as $\rho_0$, $a_0:=\left(\rho_0\right)^{-1}$
satisfy the same hypothesis and $a^n$, $\rho^n$ satisfy the same equations, they have also the same properties.

Keeping this fact in mind, let us consider each term separately.
\begin{itemize}
 \item Thanks to what we have just said, $\left(a^n\right)_n\,\subset\,\mc{C}_b([0,T]\times\mbb{R}^N)\cap L^\infty([0,T];\mc{C}^1_*)$
is bounded; moreover, from a priori estimates, we see that also $\left(\nabla\Pi^n\right)_n$ is bounded in the space
$L^\infty([0,T];\mc{C}^1_*)$. Therefore, it follows that the sequence $\left(a^n\,\nabla\Pi^n\right)_n$ is bounded in
$L^\infty([0,T];\mc{C}^{-\alpha})$.
 \item In the same way, as $\left(u^n\right)_n\subset L^\infty([0,T];\mc{C}^1_*)$ and
$\left(\nabla u^n\right)_n\subset L^\infty([0,T];L^\infty)$ are both bounded sequences, one has that the sequence
$\left(u^n\cdot\nabla u^n\right)_n$ is bounded in $L^\infty([0,T];\mc{C}^{-\alpha})$.
\end{itemize}
Therefore, exactly as done for the density, we get that $\left(u^n\right)_n$ is bounded
in $\mc{C}^{0,1}([0,T];\mc{C}^{-\alpha})$, so uniformly equicontinuous in the time variable.
This fact implies that $u^n\,\ra\,u$ in $\mc{C}([0,T];\mc{C}^{-\alpha}_{loc})$.

Finally, thanks to uniform bounds and Fatou's property of Besov spaces, we have that $u\in L^\infty([0,T];\mc{C}^1_*)$ and, by
interpolation, that $u^n\,\ra\,u$ in $\mc{C}([0,T];\mc{C}^{1-\eta}_{loc})$ for all $\eta>0$.

\medskip
So, thanks to strong convergence properties, if we test the equations on a
$\vphi\in\mc{C}^1([0,T];\mc{S}(\R^N))$ (here we have set $\mc{S}$ to be the Schwartz class), we can pass to the limit and get that
$\left(\rho,u,\nabla\Pi\right)$ is indeed a solution of the Euler system \eqref{eq:ddeuler}.

\medskip
Before going on with the striated regularity, let us establish continuity properties of the solutions with respect to
the time variable.

First of all, from
$$
\d_t\rho\,=\,-\,u\cdot\nabla\rho\,,
$$
as $u\in\mc{C}([0,T];L^\infty)$ (from the convergence properties just stated) and $\nabla\rho\in L^\infty([0,T];L^\infty)$,
we obtain that $\rho\in\mc{C}^{0,1}([0,T];L^\infty)$, and the same holds for $a:=\rho^{-1}$.

Remember that $u\in L^\infty([0,T];L^p)$, $\nabla u$ and $a\in L^\infty([0,T];L^\infty)$. Moreover, as
$\nabla\Pi\in L^\infty([0,T];L^2)\cap L^\infty([0,T];L^\infty)$, it belongs also to $L^\infty([0,T];L^p)$. So, from the equation
$$
\d_tu\,=\,-\,u\cdot\nabla u\,-\,a\,\nabla\Pi\,,
$$
we get that $\d_tu\in L^1([0,T];L^p)$, therefore $u\in\mc{C}([0,T];L^p)$.

In the same way, from \eqref{eq:vort} we get that $\Omega\in\mc{C}([0,T];L^q)$, and hence the same holds true for $\nabla u$.

Now, using elliptic equation \eqref{eq:Pi} and keeping in mind the properties just proved for $\rho$ and $a$, one can see
that $\nabla\Pi\in\mc{C}([0,T];L^2)$. So, coming back to the previous equation, we discover that also $\d_tu$ belongs to the same
space.

\subsubsection{Final checking about striated regularity}

It remains us to prove that also properties of striated regularity are preserved in passing to the limit. For doing this,
we will follow the outline of the proof in \cite{D1997}.
\begin{enumerate}
 \item \emph{Convergence of the flow}

Let $\psi^n$ and $\psi$ be the flows associated respectively to $u^n$ and $u$; for all fixed $\vphi\in\mc{D}(\R^N)$, by definition
 we have:
\begin{eqnarray*}
 \left|\vphi(x)\left(\psi^n(t,x)-\psi(t,x)\right)\right| & \leq &
\int^t_0\left|\vphi(x)\left(u^n(\tau,\psi^n(\tau,x))-u(\tau,\psi(\tau,x))\right)\right|d\tau \\
& \leq & \int^t_0\left|\vphi(x)\left(u^n-u\right)(\tau,\psi^n(\tau,x))\right|\,+ \\
& & \qquad\qquad+\,\left|\vphi(x)u^n(\tau,\psi^n(\tau,x))-\vphi(x)u^n(\tau,\psi(\tau,x))\right|d\tau \\
& \leq & \int^t_0\left\|\nabla u^n\right\|_{L^\infty}\left|\vphi(x)\left(\psi^n-\psi\right)(\tau,x)\right|d\tau\,+ \\
& & \qquad\qquad\qquad\qquad\qquad+\,\int^t_0\left\|\vphi u^n-\vphi u\right\|_{L^\infty}\,d\tau\,.
\end{eqnarray*}
So, from convergence properties stated in previous part, we have that $\psi^n\,\ra\,\psi$ in the space
$L^\infty([0,T];Id+L^\infty_{loc})$. Moreover, it's easy to see that
$$
\left\|\nabla\psi^n(t)\right\|_{L^\infty}\,\leq\,c\,\exp\left(\int^t_0\left\|\nabla u^n\right\|_{L^\infty}\,d\tau\right)\,,
$$
which tells us that the sequence $\left(\psi^n\right)_n$ is bounded in $L^\infty([0,T];Id+\mc{C}^{0,1})$. Hence, finally we discover
that $\psi^n\,\ra\,\psi$ also in the spaces $L^\infty([0,T];Id+\mc{C}^{1-\eta}_{loc})$ for all $\eta>0$.

\item \emph{Regularity of $\d_{X_0}\psi$}

First of all, let us notice that, by definition,
$$
\d_{X_0(x)}\psi^n(t,x)\,=\,X^n_t(\psi^n(t,x))\,;
$$
applying proposition \ref{p:comp}, we get
\begin{equation} \label{est:X-o-psi}
\left\|\d_{X_0}\psi^n_t\right\|_{\mc{C}^\veps}\,=\,\left\|X^n_t\circ\psi^n_t\right\|_{\mc{C}^\veps}\,\leq\,
c\,\left\|\nabla\psi^n_t\right\|_{L^\infty}\,\|X^n_t\|_{\mc{C}^\veps}\,,
\end{equation}
which implies that $\left(\d_{X_0}\psi^n\right)_n$ is bounded in the space $L^\infty([0,T];\mc{C}^\veps)$. Now we note that,
for every fixed $\vphi\in\mc{D}(\R^N)$, we have
$$
\vphi\,\d_{X_0}\psi^n\,-\,\vphi\,\d_{X_0}\psi\,=\,\d_{X_0}(\vphi\psi^n-\vphi\psi)\,-\,
\left(\d_{X_0}\vphi\right)\left(\psi^n-\psi\right)\,;
$$
the second term is compactly supported, hence it converges in $L^\infty$ because of what we have already proved.
So let us focus on the first one:
$$
\d_{X_0}\left(\vphi\,\psi^n\right)-\d_{X_0}\left(\vphi\,\psi\right)\,=\,
\div\left(X_0\otimes\vphi(\psi^n-\psi)\right)\,-\,\vphi(\psi^n-\psi)\,\,\div X_0\,;
$$
decomposing both terms in paraproduct and remainder and remembering hypothesis over $X_0$, it's easy to see that
$$
\left\|\d_{X_0}\left(\vphi\,\psi^n\right)-\d_{X_0}\left(\vphi\,\psi\right)\right\|_{\mc{C}^{\veps-1}}\,\leq\,
c\,\left\|\vphi\,\psi^n-\vphi\,\psi\right\|_{\mc{C}^\veps}\,\wtilde{\|}X_0\|_{\mc{C}^\veps}\,.
$$
Therefore, from what we have just proved, $\d_{X_0}\psi^n\,\ra\,\d_{X_0}\psi$ in $L^\infty([0,T];\mc{C}^{\veps-1}_{loc})$; moreover,
by Fatou's property, one gets that $\d_{X_0}\psi\in L^\infty([0,T];\mc{C}^\veps)$ and it verifies estimate \eqref{est:X-o-psi}.
So, by interpolation, convergence occurs also in $L^\infty([0,T];\mc{C}^{\veps-\eta}_{loc})$ for all $\eta>0$.

\item \emph{Regularity of $\,X_t$}

Remembering the definition
$$ 
 X_t(x)\,:=\,\left(\d_{X_0(x)}\psi\right)(t,\psi^{-1}_t(x))\quad\mbox{ and }\quad
\div X_t\,=\,\div X_0\circ\psi^{-1}_t\,,
$$ 
from proposition \ref{p:comp} it immediately follows that $X_t$ and $\div X_t$ both belong to $\mc{C}^\veps$. Moreover, the same
proposition implies that $X^n\,\ra\,X$ in the space $L^\infty([0,T];\mc{C}^{\veps-\eta}_{loc})$ for all $\eta>0$, and the same
holds for the
divergence. In particular, we have convergence also in $L^\infty([0,T];L^\infty_{loc})$, which finally tells us that $X_t$ remains
non-degenerate for all $t\in[0,T]$, i.e. $I(X_t)\,\geq\,c\,\,I(X_0)$.

\item \emph{Striated regularity for the density and the vorticity}

Let us first prove that regularity of the density with respect to $X_t$ is preserved in time.
To simplify the presentation, we will omit the localisation by $\vphi\in\mc{D}(\R^N)$: formally, we should repeat
the same reasoning applied to prove regularity of $\d_{X_0}\psi$.
So, let us consider
$$
\d_{X^n}\rho^n\,-\,\d_X\rho\,=\,\div\left(\rho^n\,(X^n-X)\right)\,-\,\rho^n\,\div(X^n-X)\,+\,
\div\left((\rho^n-\rho)\,X\right)\,-\,(\rho^n-\rho)\,\,\div X
$$
and prove the convergence in $L^\infty([0,T];\mc{C}^{-1}_{loc})$. Using Bony's paraproduct decomposition, it's not difficult to see
that first and third terms can be bounded by
$\|\rho^n\|_{L^\infty}\,\|X^n-X\|_{L^\infty}\,+\,\|\rho^n-\rho\|_{L^\infty}\,\|X\|_{L^\infty}$, while second and last terms can be
controlled by
$\|\rho^n\|_{L^\infty}\,\,\|\div(X^n-X)\|_{\mc{C}^{\veps/2}}\,+\,\|\rho^n-\rho\|_{L^\infty}\,\|\div X\|_{\mc{C}^{\veps/2}}$, for
instance.
So, from the convergence properties stated for $\left(\rho^n\right)_n$ and $\left(X^n\right)_n$, we get that
$\d_{X^n}\rho^n\,\ra\,\d_X\rho$ in the space $L^\infty([0,T];\mc{C}^{-1}_{loc})$, as claimed. Moreover, from a priori bounds and
Fatou's property of Besov spaces, we have that $\d_X\rho\in L^\infty([0,T];\mc{C}^\veps)$ and so, by interpolation, convergence
occurs also in $L^\infty([0,T];\mc{C}^{\veps-\eta}_{loc})$ for all $\eta>0$.

Now we consider the vorticity term (again, we omit the multiplication by a $\mc{D}(\R^N)$ function):
\begin{eqnarray*}
\d_{X^n}\Omega^n\,-\,\d_X\Omega & = & \div\left((X^n-X)\otimes\Omega^n\right)\,-\,\Omega^n\,\,\div(X^n-X)\,+ \\
& & \qquad\qquad\qquad+\,\div\left(X\otimes(\Omega^n-\Omega)\right)\,-\,(\Omega^n-\Omega)\,\,\div X\,.
\end{eqnarray*}
From the convergence properties of $\left(u^n\right)_n$, we know that $\Omega^n\,\ra\,\Omega$ in $L^\infty([0,T];\mc{C}^{-\eta}_{loc})$
for all $\eta>0$, so for instance also for $\eta=\veps/2$. From this, using again paraproduct decomposition as done before, one can
prove that $\d_{X^n}\Omega^n\,\ra\,\d_X\Omega$ in $L^\infty([0,T];\mc{C}^{-1-\veps/2}_{loc})$. From a priori
estimates and Fatou's property of Besov spaces again, this implies $\d_X\Omega\in L^\infty([0,T];\mc{C}^{\veps-1})$, and moreover
convergence remains true (by interpolation) in spaces $L^\infty([0,T];\mc{C}^{\veps-1-\eta}_{loc})$ for all $\eta>0$.
\end{enumerate}

So, all the properties linked to striated regularity are now verified, and this concludes the proof of the existence part
of theorem \ref{t:stri-N}.

\subsection{Uniqueness}

Uniqueness in theorem \ref{t:stri-N} is an immediate consequence of the following stability result.

\begin{prop} \label{p:stab}
 Let $\left(\rho^1,u^1,\nabla\Pi^1\right)$ and $\left(\rho^2,u^2,\nabla\Pi^2\right)$ be solutions of system \eqref{eq:ddeuler} with
$$
0\,\,<\,\,\rho_*\,\,\leq\,\,\rho^1\,,\,\rho^2\,\,\leq\,\,\rho^*\,.
$$
Let us suppose that $\delta\rho:=\rho^1-\rho^2\,\in\,\mc{C}([0,T];L^2)$ and that
$\delta u:=u^1-u^2\,\in\,\mc{C}^1([0,T];L^2)$. \\
Finally, assume that $\nabla\rho^2$, $\nabla u^1$, $\nabla u^2$ and $\nabla\Pi^2$ all belong to $L^1([0,T];L^\infty)$.

Then, for all $t\in[0,T]$, we have the following estimate:
$$
\left\|\delta\rho(t)\right\|_{L^2}\,+\,\left\|\delta u(t)\right\|_{L^2}\,\leq\,C\,e^{c\,I(t)}\,
\left(\left\|\delta\rho(0)\right\|_{L^2}\,+\,\left\|\delta u(0)\right\|_{L^2}\right)\,,
$$
where we have defined
$$
I(t)\,:=\,\int^t_0\left(\left\|\nabla\rho^2\right\|_{L^\infty}+\left\|\nabla u^1\right\|_{L^\infty}+
\left\|\nabla u^2\right\|_{L^\infty}+\left\|\nabla\Pi^2\right\|_{L^\infty}\right)d\tau\,.
$$
\end{prop}

\begin{proof}
 From $\d_t\delta\rho\,+\,u^1\cdot\nabla\delta\rho\,=\,-\,\delta u\cdot\nabla\rho^2$, we immediately get
$$
\|\delta\rho(t)\|_{L^2}\,\leq\,\|\delta\rho(0)\|_{L^2}\,+\,
\int^t_0\|\delta u\|_{L^2}\,\left\|\nabla\rho^2\right\|_{L^\infty}\,d\tau\,.
$$
Moreover, the equation for $\delta u$ reads as follows:
$$
\d_t\delta u\,+\,u^1\cdot\nabla\delta u\,=\,-\,\delta u\cdot\nabla u^2\,-\,\frac{\nabla\delta\Pi}{\rho^1}\,+\,
\frac{\nabla\Pi^2}{\rho^1\,\rho^2}\,\delta\rho\,,
$$
where we have set $\delta\Pi\,=\,\Pi^1-\Pi^2$. So, from standard $L^p$ estimates for transport equations, one has:
$$
\|\delta u(t)\|_{L^2}\,\leq\,\|\delta u(0)\|_{L^2}+C\int^t_0\left(\|\delta u\|_{L^2}\left\|\nabla u^2\right\|_{L^\infty}+
\|\nabla\delta\Pi\|_{L^2}+\left\|\nabla\Pi^2\right\|_{L^\infty}\|\delta\rho\|_{L^2}\right)d\tau\,.
$$
Now, let us analyse the equation for $\nabla\delta\Pi$:
\begin{eqnarray*}
-\,\div\left(\frac{\nabla\delta\Pi}{\rho^1}\right) & = & \div\left(-\,\frac{\nabla\Pi^2}{\rho^1\,\rho^2}\,\delta\rho\,+\,
u^1\cdot\nabla\delta u\,+\,\delta u\cdot\nabla u^2\right) \\
& = & \div\left(-\,\frac{\nabla\Pi^2}{\rho^1\,\rho^2}\,\delta\rho\,+\,\delta u\cdot(\nabla u^1+\nabla u^2)\right)\,,
\end{eqnarray*}
where, to get the second equality, we have used the algebraic identity
$$
\div(v\cdot\nabla w)\,=\,\div(w\cdot\nabla v)\,+\,\div(v\,\,\div w)\,-\,\div(w\,\,\div v)\,.
$$
So, from lemma \ref{l:laxmilgram} we obtain
$$
\|\nabla\delta\Pi\|_{L^2}\,\leq\,C\left(\left\|\nabla\Pi^2\right\|_{L^\infty}\,\|\delta\rho\|_{L^2}\,+\,
\|\delta u\|_{L^2}\left(\left\|\nabla u^1\right\|_{L^\infty}+\left\|\nabla u^2\right\|_{L^\infty}\right)\right)\,,
$$
and Gronwall's inequality completes the proof of the proposition.
\end{proof}

Now, let us prove uniqueness part in theorem \ref{t:stri-N}.
Let $\left(\rho^1,u^1,\nabla\Pi^1\right)$ and $\left(\rho^2,u^2,\nabla\Pi^2\right)$ satisfy
system \eqref{eq:ddeuler} with same initial data $\left(\rho_0,u_0\right)$, under hypothesis of theorem \ref{t:stri-N}.

As $\delta u(0)=0$ and $u\in\mc{C}([0,T];L^p)$, $\nabla u\in\mc{C}([0,T];L^q)$, one easily gets that
$\delta u\in\mc{C}^1([0,T];L^2)$.
Moreover, from this fact, observing that also $\delta\rho(0)=0$, the equation for $\delta\rho$ tells us that
$\delta\rho\in\mc{C}([0,T];L^2)$.
Hence proposition \ref{p:stab} can be applied and uniqueness immediately follows.

\section{On the lifespan of the solution} \label{s:lifespan}

The aim of this section is to establish, in the most accurate way, an explicit lower bound for the lifespan of the solution of system
\eqref{eq:ddeuler} in terms of initial data only.

For notation convenience, let us define
$$
L_0\,:=\,\|u_0\|_{L^p}+\|\Omega_0\|_{L^q\cap L^\infty}\qquad\mbox{ and }\qquad A_0\,:=\,\|\nabla\rho_0\|_{L^\infty}\,.
$$

\begin{theo} \label{t:life}
Under the hypothesis of theorem \ref{t:stri-N},
the lifespan $T$ of a solution to system \eqref{eq:ddeuler} with initial data $(\rho_0,u_0)$ is bounded from below,
up to multiplication by a constant (depending only on the space dimension $N$, $\veps$, $p$, $q$, $\rho_*$ and $\rho^*$),
by the quantity
\begin{equation} \label{life:T}
\frac{\min\left\{L_0,\left\|\Omega_0\right\|_{\mc{C}^\veps_{X_0}}\right\}\times\left(L_0
\log\left(e+\frac{\left\|\Omega_0\right\|_{\mc{C}^\veps_{X_0}}}{L_0}\right)\right)^{-1}}
{\left(1+L_0+\left\|\Omega_0\right\|_{\mc{C}^\veps_{X_0}}\right)^{\!\!2}\left(1+A_0^{\delta+3}\right)
\left(1+\wtilde{|||}X_0|||^3_{\mc{C}^\veps}+\left\|\d_X\nabla\rho_0\right\|_{\mc{C}^\veps_{X_0}}\right)}\,,
\end{equation}
where $\delta>1$ is the exponent which occurs in \eqref{est:Pi_C^1_*}.
\end{theo}

\begin{proof}
 Our starting point is subsection \ref{ss:fin-est}: with the same notations,
moreover we define
\begin{eqnarray*}
& & \Theta(t):=L(t)\,\log\left(e+\frac{S(t)}{L(t)}\right)\,,\;\;
U(t):=\int^t_0\left\|\nabla u(\tau)\right\|_{L^\infty}\,d\tau\,,\\
& & A(t):=\left\|\nabla\rho(t)\right\|_{L^\infty}\,,\;\;
\Gamma(t):=\wtilde{\|}X(t)\|_{\mc{C}^\veps}\,,\;\;R(t):=\left\|\d_{X(t)}\nabla\rho(t)\right\|_{\mc{C}^{\veps-1}}\,.
\end{eqnarray*}
It's only matter of repeating previous computations in a more accurate way.

Let us notice that, from inequality \eqref{est:Du_L^inf}, for all time $t$ one has
\begin{equation} \label{est:U}
U'(t)\,=\,\|\nabla u(t)\|_{L^\infty}\,\leq\,C\,\Theta(t)\,:
\end{equation}
we will make a broad use of this fact.

Now, let us define the time $T_1\,:=\,\sup\left\{t>0\,|\,U(t)\,\leq\,\log2\right\}$. Then, on $[0,T_1]$ we have
(from \eqref{est:rho_L^inf} and \eqref{est:Drho_L^inf})
$$
A(t)\,\leq\,C\,A_0\qquad\mbox{ and }\qquad \|\rho(t)\|_{W^{1,\infty}}\,\leq\,C\,\|\rho_0\|_{W^{1,\infty}}\,,
$$
and so, keeping in mind \eqref{est:u_L^p}, \eqref{est:Om_L^q}, \eqref{est:Om_L^inf} and \eqref{f-est:Pi}, we get also
\begin{eqnarray}
 L(t) & \leq & C\left(L(0)+\left(1+A^{\delta+1}_0\right)\int^t_0\Theta^2(\tau)\,d\tau\right) \label{life-est:L} \\
\|\nabla\Pi\|_{L^2\cap\mc{C}^1_*} & \leq & C\left(1+A^{\delta+1}_0\right)\,\Theta^2(t)\,. \nonumber
\end{eqnarray}
In addition, \eqref{est:I} implies $I(X(t))\geq CI(X_0)$, while from \eqref{est:X_C^e} and \eqref{est:d_X-u}
it follows
$$
\Gamma(t)\,\leq\,C\left(\Gamma_0+\int^t_0S(\tau)d\tau\right)\qquad\mbox{ and }\qquad
\left\|\d_Xu(t)\right\|_{\mc{C}^\veps}\,\leq\,C\left(S(t)\,+\,\Gamma(t)\,\Theta(t)\right).
$$
Finally, \eqref{est:d_X-rho} and \eqref{est:d_X-Drho} entail
$$
\left\|\d_X\rho\right\|_{\mc{C}^\veps}\,\leq\,C\left(\bigl(1+A_0\bigr)\Gamma_0+R_0\right)\;\mbox{ and }\;
\left\|\d_X\nabla\rho\right\|_{\mc{C}^{\veps-1}}\,\leq\,C\left(\bigl(1+A_0\bigr)\Gamma_0+R_0+\int^t_0S(\tau)d\tau\right).
$$

From the inequalities we've just established, the control of the striated norm of $\nabla\Pi$ immediately follows.

Let us proceed carefully, as done in subsection \ref{ss:fin-est}. After some simply (even if rough) manipulations, we get
(up to multiplication by constant terms)
\begin{eqnarray*}
 \|\rho\|_{W^{1,\infty}}\left\|\d_X\rho\right\|_{\mc{C}^\veps}\|\nabla\Pi\|_{\mc{C}^1_*} & \leq &
\left(1+A^{\delta+2}_0\right)\left(\Gamma_0+R_0\right)\Theta^2(t) \\
\|\nabla\Pi\|_{\mc{C}^1_*}\wtilde{\|}X\|_{\mc{C}^\veps}\|\rho\|^2_{W^{1,\infty}} & \leq & 
\left(1+A^{\delta+2}_0\right)\left(\Gamma_0+\int^t_0S(\tau)d\tau\right)\Theta^2(t)\,.
\end{eqnarray*}
Now, thanks to $(a+b)^3\leq C(a^3+b^3)$ and Jensen's inequality we infer
$$
\wtilde{\|}X\|^3_{\mc{C}^\veps}\|\nabla\Pi\|_{\mc{C}^1_*}\,\leq\,\left(\Gamma_0+
\int^t_0S^3(\tau)d\tau\right)\left(1+A^{\delta}_0\right)\Theta^2(t)\,.
$$
Finally, the fact that $\|u(t)\|_{\mc{L}^{p,\infty}}\leq\Theta(t)$ implies
\begin{eqnarray*}
 \|\rho\|_{W^{1,\infty}}\wtilde{\|}X\|_{\mc{C}^\veps}\|u\|^2_{\mc{L}^{p,\infty}} & \leq & 
(1+A_0)\left(\Gamma_0+\int^t_0S(\tau)\,d\tau\right)\Theta^2(t) \\
\|\rho\|_{W^{1,\infty}}\|u\|_{\mc{L}^{p,\infty}}\left\|\d_Xu\right\|_{\mc{C}^\veps} & \leq &
(1+A_0)\,\Theta(t)\left(S(t)+\left(\Gamma_0+\int^t_0S(\tau)\,d\tau\right)\Theta(t)\right) \\
& \leq & (1+A_0)\,(1+\Gamma_0)\,(S(t)+1)\,\Theta^2(t)\,+\,(1+A_0)\,\Theta^2(t)\int^t_0S(\tau)d\tau\,,
\end{eqnarray*}
where, in deriving the last bound, we have used also that $\Theta(t)\geq1$.

Let us define
$$
M_0\,:=\,\left(1\,+\,A_0^{\delta+2}\right)\,\left(1\,+\,\Gamma_0^3\,+\,R_0\right);
$$
as $\int S\leq1+\int S^3$, in the end we get
\begin{equation} \label{life-est:Pi}
\left\|\d_{X(t)}\nabla\Pi(t)\right\|_{\mc{C}^\veps}\,\leq\,C\,M_0\,\Theta^2(t)
\left(1+S(t)+\int^t_0S^3(\tau)d\tau\right).
\end{equation}

Now let us focus on the striated norm of the vorticity, estimated in \eqref{est:d_X-Om}.
Analysing each term which occurs in the definition \eqref{def:iota} of $\Upsilon$, we see that first,
second and fourth items can be bounded by $\left\|\d_{X(t)}\nabla\Pi(t)\right\|_{\mc{C}^\veps}$, and the third one
is controlled as in \eqref{life-est:Pi}, up to replace $M_0$ with $\wtilde{M}_0:=(1+A_0)M_0$.  Finally, we have
\begin{eqnarray*}
 \|\rho\|_{W^{1,\infty}}\left\|\nabla\Pi\right\|_{\mc{C}^1_*}\left\|\d_X\nabla\rho\right\|_{\mc{C}^{\veps-1}} & \leq &
(1+A_0)\,\left(1+A^{\delta}_0\right)\,\Theta^2(t)\left((1+A_0)\Gamma_0+R_0+\int^t_0S(\tau)d\tau\right) \\
& \leq & \left(1+A^{\delta+2}_0\right)\left(1+\Gamma_0+R_0\right)\Theta^2(t)\left(1+\int^t_0S(\tau)d\tau\right).
\end{eqnarray*}

So, putting all these inequalitites together, we discover that, in $[0,T_1]$,
$$
S(t)\,\leq\,C\left(S_0+\wtilde{M}_0\int^t_0\Theta^2(\tau)\left(1+S(\tau)+\int^\tau_0S^3(\tau')d\tau'\right)d\tau\right),
$$
and by Gronwall's lemma this finally implies
$$
S(t)\,\leq\,C\,e^{c\int^t_0\Theta^2(\tau)d\tau}\left(S_0+
\wtilde{M}_0\int^t_0e^{-\int^\tau_0\Theta^2(\tau')d\tau'}\Theta^2(\tau)\left(1+\int^\tau_0S^3(\tau')d\tau'\right)d\tau\right).
$$

Let us now define $T_2$ as the supremum of the times $t>0$ for which both the conditions
$$
\wtilde{M}_0\int^t_0\Theta^2(\tau)d\tau\leq2 L_0\quad\mbox{ and }\quad
\wtilde{M}_0\int^t_0
\Theta^2(\tau)\left(1+\int^\tau_0S^3(\tau')d\tau'\right)d\tau\leq2 S_0
$$
are fulfilled. Note that, as $\Theta\geq1$, also \eqref{est:U} is verified in $[0,T_2]$, so $T_2\leq T_1$.
Therefore, keeping in mind \eqref{life-est:L}, in $[0,T_2]$ one has
$$
S(t)\,\leq\,C\,S_0\,,\qquad L(t)\,\leq\,C\,L_0\qquad\mbox{ and }\qquad \Theta(t)\,\leq\,C\,\Theta_0\,,
$$
because the function $(\lambda,\sigma)\mapsto\lambda\log(e+\sigma/\lambda)$ is increasing both in $\lambda$ and $\sigma$.

Let us put these bounds in the integral condition defining $T_2$: we discover that $T_2$ is greater than or equal to
every time $t$ for which
$$
\wtilde{M}_0\,\Theta^2_0\,t\,\leq\,2\,L_0\quad\mbox{ and }\quad
\wtilde{M}_0\,\Theta^2_0\,t\,+\,\wtilde{M}_0\,\Theta^2_0\,S_0^3\,\frac{t^2}{2}\,\leq\,2\,S_0\,.
$$
Therefore, if we define
\begin{equation} \label{life-def:T}
 T\;:=\;\frac{K\,\min\{L_0,S_0\}}{\wtilde{M}_0\,(1+L_0+S_0)^2}\,\Theta_0^{-1}\,,
\end{equation}
then both the previous relations are fulfilled, for some suitable constant $K$. Hence, $T\leq T_2$, and the theorem is now proved.
\end{proof}

\begin{rem} \label{r:life}
Let us notice that, in the classical case (constant density), the lifespan of a solution was controlled from below by
$$
T_{cl}\,:=\,C\left(\|\Omega_0\|_{L^q\cap L^\infty}\,
\log\left(e+\frac{\|\Omega_0\|_{\mc{C}^\veps_{X_0}}}{\|\Omega_0\|_{L^q\cap L^\infty}}\right)\right)^{-1}
$$
(see also \cite{D1999}). We have just proved that in our case the lifespan is given by \eqref{life:T}, instead.
The two lower bounds are quite similar, even if in our case also the initial density comes into play, and
there are some additional items, basically due to the more complicate analysis of the pressure term.
\end{rem}

\begin{rem} \label{r:life_2}
 Note also that, in the two dimensional case, the stretching term in the vorticity equation disappears. This fact
translates, at the level of a priori estimates, into the absence of the first two items in the right-hand side of \eqref{def:iota}.
Nevertheless, as we have seen, the analysis of $\nabla\Pi$ produces terms of this kind: for this reason,
in dimension $N=2$ we weren't able to improve the lower bound \eqref{life:T}.
\end{rem}

\section{``H\"older continuous vortex patches''} \label{s:conormal}

First of all, let us prove conservation of conormal regularity.

Given a compact hypersurface $\Sigma\subset\R^N$ of class $\mc{C}^{1,\veps}$, we can always find, in a canonical way,
a family $X$ of vector-fields such that the inclusion $\,\mc{C}^\eta_\Sigma\,\subset\,\mc{C}^\eta_X$ holds
for all $\eta\in[\veps,1+\veps]$. For completeness, let us  recall the result (see proposition 5.1 of \cite{D1999}), which
turns out to be important in the sequel.
\begin{prop} \label{p:con->stri}
Let $\Sigma$ be a compact hypersurface of class $\mc{C}^{1,\veps}$.

Then there exists a non-degenerate family of $m=N(N+1)/2$ vector-fields $X\subset\mc{T}^\veps_\Sigma$ such that
$\mc{C}^\eta_\Sigma\,\subset\,\mc{C}^\eta_X$ for all $\eta\in[\veps,1+\veps]$.
\end{prop}
Hence, thanks to theorem \ref{t:stri-N} we propagate striated regularity with respect to this family.
Finally, in a classical way, from this fact one can recover conormal properties of the solution, and so get the thesis of
theorem \ref{t:conorm-N} (see e.g. \cite{G-SR}, sections 5 and 6, and \cite{D1999}, section 5, for the details).

\medskip
Actually, in the case of space dimension $N=2\,,\,3$ (finally, the only relevant ones from the physical point of view) one can
improve the statement of theorem \ref{t:conorm-N}. To avoid traps coming from differential geometry, let us clarify our
work setting.

In considering a submanifold $\Sigma\subset\R^N$ of dimension $k$ and of class $\mc{C}^{1,\veps}$ (for some
$\veps>0$), we mean that $\Sigma$ is a manifold of dimension $k$ endowed
with the differential structure inherited from its inclusion in $\R^N$, and the transition maps are of class $\mc{C}^{1,\veps}$. \\
In particular, for all $x\in\Sigma$ there is an open ball $B\subset\R^N$ containing $x$, and a $\mc{C}^{1,\veps}$ local parametrization
$\vphi:\R^k\ra B\cap\Sigma$ with inverse of class $\mc{C}^{1,\veps}$. This is equivalent to require local equations
$H:B\ra\R^k$ of class $\mc{C}^{1,\veps}$ such that $H_{|B\cap\Sigma}\equiv0$.

Let us explicitly point out that, when we speak about generic submanifolds, we always mean submanifolds without boundary.

Given a local parametrization $\vphi$ on $U:=\Sigma\cap B$, its differential $\vphi_*:T\R^k\ra TU\cong T\Sigma$ induces, in
each point $x\in\R^k$, a linear isomorphism between the tangent spaces,
$\vphi_{*,x}:T_x\R^k\ra T_{\vphi(x)}\Sigma$. Moreover, the dependence of this map on the point $x\in\R^k$
is of class $\mc{C}^\veps$: in coordinates, $\vphi_*$ is given by the Jacobian matrix $\nabla\vphi$.

Finally, we say that a function $f$ defined on $\Sigma$ is (locally) of class $\mc{C}^\alpha$ (for $\alpha>0$) if
the composition $f\circ\vphi:\R^k\ra\R$ is $\alpha$-H\"older continuous for any local parametrization $\vphi$.

\medskip
Before stating our claim, some preliminary results are in order. Let us start with a very simple lemma.
\begin{lemma} \label{l:D->Hold}
 Let $f\in L^\infty(\R^N)$ such that its gradient is $\alpha$-H\"older continuous for some $\alpha>0$.

Then $f\in\mc{C}^{1,\alpha}(\R^N)\equiv B^{1+\alpha}_{\infty,\infty}(\R^N)$.
\end{lemma}

\begin{proof}
 It's obvious using dyadic characterization of H\"older spaces and Bernstein's inequalities.
\end{proof}

Now, by analogy, one may ask if this property still holds true for a function defined on a submanifold, with H\"older continuous
tangential derivatives. The answer is yes, with some additional hypothesis on the submanifold.
\begin{prop} \label{p:Hold-man}
 Let $\Sigma\subset\R^N$ be a submanifold of dimension $k$ and of class $\mc{C}^{1,\veps}$, for some $\veps>0$. Moreover, let us
suppose $\Sigma$ to be compact. \\
Let us consider a function $f:\Sigma\ra\R$, bounded on $\Sigma$, such that
$\d_Xf\in\mc{C}^\veps(\Sigma)$ for all vector-fields $X$ of class $\mc{C}^\veps$ tangent to $\Sigma$.

Then $f\in\mc{C}^{1,\veps}(\Sigma)$.
\end{prop}

\begin{proof}
 Let us fix a coordinate set $U:=B\cap\Sigma$ (for some open ball $B\subset\R^N$) with its $\mc{C}^{1,\veps}$ local parametrization
$\vphi:\R^k\ra U$, and let us define $g:=f\circ\vphi\,:\R^k\ra\R$.

Obviously, $g\in L^\infty(\R^k)$, because $f\in L^\infty(\Sigma)$.

Moreover, for all $1\leq i\leq k$ let us set $\vphi_*(\d_i)=X_i\,$: then, $X_i$ is obviously of class $\mc{C}^\veps$.
Hence we have $\d_ig(x)=X_i(f)\!\left(\vphi(x)\right)$, i.e. $\d_ig$ at a point $x$ is the derivation $X_i$ applied to
the function $f$, and evaluated at the point $\vphi(x)$. In our notations, we get $\d_ig=(\d_{X_i}f)\circ\vphi$.

Therefore, from our hypothesis it follows that $\nabla g\in\mc{C}^\veps$, and so, by lemma \ref{l:D->Hold},
$g\in\mc{C}^{1,\veps}(\R^k)$.

In conclusion, we have proved that $f$ composed with any local parametrization $\vphi$ is of class $\mc{C}^{1,\veps}$
on $\R^k$. Therefore $f\in\mc{C}^{1,\veps}(\Sigma)$, and, as $\Sigma$ is compact, we can bound its H\"older norm globally.
\end{proof}

\begin{rem} \label{r:Hold-man}
 Let us note that the operator $\d_X$ depends linearly on the vector-field $X$. Hence, in the hypothesis of the previous
proposition it's enough to assume that one can find, locally on $\Sigma$,
a family $\left\{X_1,\ldots,X_k\right\}$ of linearly independent vector-fields of class $\mc{C}^\veps$ such that
$\d_{X_i}f\in\mc{C}^\veps(\Sigma)$ for all $1\leq i\leq k$.
\end{rem}

\begin{coroll} \label{c:Hold-man}
 Let $\Sigma\subset\R^N$ be a compact hypersurface of class $\mc{C}^{1,\veps}$, and let $f\in\mc{C}^\veps(\R^N)$.

If $f\in\mc{C}^{1+\veps}_\Sigma$, then $f_{|\Sigma}\in\mc{C}^{1,\veps}(\Sigma)$.
\end{coroll}

\begin{proof}
 By proposition \ref{p:con->stri} and non-degeneracy condition, we can find, locally on $\Sigma$, $N-1$ linearly independent
vector-fileds $X_1\ldots X_{N-1}$, defined on the whole $\R^N$ and of class $\mc{C}^\veps$, which are tangent to $\Sigma$ and
such that $\div\!\left(f\,X_i\right)\in\mc{C}^\veps(\R^N)$ for all $1\leq i\leq N-1$.

Moreover, also the divergence of these vector-fields is $\veps$-H\"older continuous; therefore, using also Bony's paraproduct
decomposition, we gather that
$$
\d_{X_i}f\,=\,\div\!\left(f\,X_i\right)\,-\,f\,\,\div X_i\;\,\in\;\,\mc{C}^\veps(\R^N)\qquad\forall\,\,1\leq i\leq N-1\,,
$$
and hence this regularity is preserved if we restrict $\d_{X_i}f$ only to $\Sigma$.

So, proposition \ref{p:Hold-man} and remark \ref{r:Hold-man} both imply that $f_{|\Sigma}\in\mc{C}^{1,\veps}(\Sigma)$.
\end{proof}

Now, let us come back to the situation of theorem \ref{t:conorm-N}. Moreover, let us suppose that the hypersurface
$\Sigma_0$ is also connected: then it separates the whole space $\R^N$ into two connected components,
the first one bounded and the other one unbounded, and whose boundary is exactly $\Sigma_0$.
In dimension $2$, this is nothing but the Jordan curve theorem, while in
the general case $N\geq3$ it's a consequence of the Alexander duality theorem (see e.g. \cite{Hat}, theorem 3.44).
For the sake of completeness, we will quote the exact statement and its proof
in appendix \ref{app:alg_top}.

So, let us set $D_0$ to be the bounded domain of $\R^N$ whose
boundary is $\d D_0=\Sigma_0$ and let us define $D(t)=\psi_t(D_0)$. As the flow $\psi_t$ is a diffeomorphism for every fixed time $t$,
we have that $\d D(t)=\Sigma(t)$ and also the complementary region is transported by $\psi$: $D(t)^c=\psi_t(D^c_0)$.

Let us denote by $\chi_\mc{O}$ the characteristic function of a set $\mc{O}$.

\begin{theo} \label{t:v-patches}
  Under hypothesis of theorem \ref{t:conorm-N}, suppose also that the initial data can be decomposed in the following way:
$$
 \rho_0(x)\,=\,\rho^i_0(x)\,\chi_{D_0}(x)\,+\,\rho^e_0(x)\,\chi_{D^c_0}(x)\quad\mbox{ and }\quad
\Omega_0(x)\,=\,\Omega^i_0(x)\,\chi_{D_0}(x)\,+\,\Omega^e_0(x)\,\chi_{D^c_0}(x)\,,
$$
with $\rho^i_0\in\mc{C}^{1,\veps}(D_0)$ and $\Omega^i_0\in\mc{C}^\veps(D_0)$.

Then, the previous decomposition still holds for the solution at every time $t\in[0,T]$:
\begin{eqnarray}
 \rho(t,x) & = & \rho^i(t,x)\,\chi_{D(t)}(x)\,+\,\rho^e(t,x)\,\chi_{D(t)^c}(x) \label{eq:dec_rho} \\
\Omega(t,x) & = & \Omega^i(t,x)\,\chi_{D(t)}(x)\,+\,\Omega^e(t,x)\,\chi_{D(t)^c}(x)\,. \label{eq:dec_vort}
\end{eqnarray}
Moreover, H\"older continuity in the interior of the domain $D(t)$ is preserved, uniformly
on $[0,T]$: at every time $t$, we have
$$
\rho^i(t)\,\in\,\mc{C}^{1,\veps}(D(t))\qquad\mbox{ and }\qquad
\Omega^i(t)\,\in\,\mc{C}^\veps(D(t))
$$
and regularity on $D(t)$ propagates also for the velocity field and the pressure term: $u(t)$ and
$\nabla\Pi(t)$ both belong to $\mc{C}^{1,\veps}(D(t))$.
\end{theo}

\begin{proof}
First of all, let us recall that, by theorem \ref{t:conorm-N}, on $[0,T]$ we have
\begin{equation} \label{est:L^1_t-Du}
 \int^T_0\left\|\nabla u(t)\right\|_{L^\infty}\,dt\,\leq\,C\,.
\end{equation}

Thanks to first equation of \eqref{eq:ddeuler}, relation \eqref{eq:dec_rho} obviously holds, with
$$
\rho^{i,e}(t,x)\,=\,\rho^{i,e}_0\left(\psi^{-1}_t(x)\right)\,.
$$
So, we immediately get that $\rho^i(t)$ belongs to the space $\mc{C}^{1,\veps}(D(t))$.
Let us observe also that a decomposition analogous to \eqref{eq:dec_rho} holds also for $a=1/\rho$, and its components
$a^{i,e}$ have the same properties of the corresponding ones of $\rho$.

Now let us handle the vorticity term. We can always decompose the solution in a component localised on $D(t)$ and the other one
supported on the complementary set, defining
$$
\Omega^i(t,x)\,:=\,\Omega(t,x)\,\chi_{D(t)}(x)\;,\qquad\Omega^e(t,x)\,:=\,\Omega(t,x)\,\chi_{D(t)^c}(x)\,,
$$
and therefore obtain relation \eqref{eq:dec_vort}. By virtue of this fact, equation \eqref{eq:vort} restricted on the domain
$D(t)$ reads as follows:
$$
\d_t\Omega^i\,+\,u\cdot\nabla\Omega^i\,=\,-\left(\Omega^i\cdot\nabla u\,+\,^t\nabla u\cdot\Omega^i\,+\,
\nabla a^i\wedge\nabla\Pi\right),
$$
which gives us the estimate (keep in mind also \eqref{est:L^1_t-Du})
$$
 \left\|\Omega^i(t)\right\|_{\mc{C}^\veps}\,\leq\,C\,
\biggl(\left\|\Omega^i_0\right\|_{\mc{C}^\veps}\,+\,\int^t_0
\left(\left\|\Omega^i\cdot\nabla u+\,^t\nabla u\cdot\Omega^i\right\|_{\mc{C}^\veps}+
\left\|\nabla a^i\wedge\nabla\Pi\right\|_{\mc{C}^\veps}\right)d\tau\biggr)\,.
$$
We claim that the first term under the integral can be controlled in $\mc{C}^\veps$. As a matter of facts,
by \eqref{eq:BS-law} we know that the velocity field satisfies the elliptic equation
$$
-\,\Delta u^k\,=\,\sum_{j=1}^N\d_j\Omega^i_{kj}
$$
in $D(t)$, with the boundary condition (by theorem \ref{t:conorm-N} and corollary \ref{c:Hold-man})
$u_{|\d D(t)}\in\mc{C}^{1,\veps}(\d D(t))$. So
(see theorem 8.33 of \cite{G-T}) we have that $u_{|D(t)}\in\mc{C}^{1,\veps}(D(t))$ and the following inequality holds:
$$
\|u\|_{\mc{C}^{1,\veps}(D(t))}\,\leq\,C\left(\|u(t)\|_{L^\infty(D(t))}\,+\,
\left\|u_{|\d D(t)}\right\|_{\mc{C}^{1,\veps}(\d D(t))}\,+\,\left\|\Omega^i\right\|_{\mc{C}^\veps(D(t))}\right).
$$
Let us note that, as pointed out in \cite{G-T}, a priori the constant $C$ depends on $\d D(t)$ through the
$\mc{C}^{1,\veps}$ norms of its local parametrizations, so finally on $\exp\left(\int^t_0\|\nabla u\|_{L^\infty}d\tau\right)$.
However relation \eqref{est:L^1_t-Du} allows us to control it uniformy on $[0,T]$.
Therefore, in $D(t)$ one gets the following inequality:
$$
\left\|\Omega^i\cdot\nabla u+\,^t\nabla u\cdot\Omega^i\right\|_{\mc{C}^\veps(D(t))}\,\leq\,C\,
\left\|\Omega^i\right\|_{\mc{C}^\veps(D(t))}\,\|u_{|D(t)}\|_{\mc{C}^{1,\veps}(D(t))}\,,
$$
which proves our claim.

Finally, let us handle the pressure term.
From what we have proved, $\nabla a^i$ is in $\mc{C}^\veps$; so
$$
\left\|\nabla a^i\wedge\nabla\Pi\right\|_{\mc{C}^\veps}\,\leq\,C\,\left\|\nabla a^i\right\|_{\mc{C}^\veps}\,
\|\nabla\Pi\|_{\mc{C}^1_*(\R^N)}\,.
$$
However, we want to prove that an improvement of regularity in the interior of $D(t)$ occurs also for $\nabla\Pi$.
In fact, keeping in mind \eqref{eq:Lapl-Pi}, $\Pi$ satisfies the elliptic equation
$$
-\,\Delta\Pi\,=\,\nabla(\log \rho^i)\cdot\nabla\Pi\,+\,\rho^i\,\nabla u:\nabla u
$$
in the bounded domain $D(t)$. Now, from what we have proved, the right-hand side obviously belongs to $\mc{C}^\veps(D(t))$.
Moreover, by theorem \ref{t:conorm-N}
and corollary \ref{c:Hold-man}, we have $\nabla\Pi_{|\d D(t)}\in\mc{C}^{1,\veps}(\d D(t))$: in particular, as $\Sigma(t)$
is compact, $\Pi_{|\d D(t)}$ is continuous. Finally, as $D(t)$ is of class $\mc{C}^{1,\veps}$, it satisfies the exterior
cone condition (see \cite{Dau-Lions}, page 340). So, theorem 6.13 of \cite{G-T} applies: from it, we gather
$\Pi(t)\in\mc{C}^{2,\veps}(D(t))$. Therefore, $\nabla\Pi(t)_{|D(t)}\in\mc{C}^{1,\veps}(D(t))$ and its norm is bounded
by
$$
\left\|\nabla\Pi_{|\d D(t)}\right\|_{\mc{C}^{1,\veps}(\d D(t))}+
\left\|\nabla a^i\right\|_{\mc{C}^\veps(D(t))}\left\|\nabla\Pi\right\|_{\mc{C}^1_*(\R^N)}+
\left\|\rho^i\right\|_{\mc{C}^{1,\veps}(D(t))}\left\|\nabla u\right\|^2_{\mc{C}^\veps(D(t))}\,.
$$

Putting all these inequalities together and applying Gronwall's lemma, we finally get a control for the $\mc{C}^\veps$ norm of
$\Omega^i$ in the interior of $D(t)$, and this completes the proof of the theorem.
\end{proof}


\appendix

\section{Appendix -- Complements from Algebraic Topology} \label{app:alg_top}

Here we want to prove the following theorem, which we used in section \ref{s:conormal}.
For the technical definitions, notions and results, we refer to \cite{Hat}.
\begin{theo} \label{th:piana}
 For any dimension $N\geq2$, let $\Sigma\subset\R^N$ be a compact, connected hypersurface without boundary.

Then $\R^N\!\setminus\!\Sigma$ has two connected components (say) $B$ and $U$, one bounded
and the other one unbounded, whose boundary is just $\Sigma$.
\end{theo}

The previous result implies in particular that $\Sigma$ is orientable (see theorem \ref{th:orient}). \\
Let us note that, if we already assumed this (redundant) hypothesis in theorem \ref{th:piana}, then the
proof would be easier (see e.g. \cite{Lima}).

The proof we quote here is actually due to A. Lerario.

\begin{proof}
With standard notations, for a submanifold $\mc{M}\subset\R^N$ and an abelian group $G$, we denote with
$$
\wtilde{H}_k(\mc{M};G)\qquad\mbox{ and }\qquad\wtilde{H}^k(\mc{M};G)
$$
the $k$-th reduced homology and cohomology groups of $\mc{M}$ with coefficients in $G$.

\medbreak
Let us compactify $\R^N$ by adding the point at infinity: in this way, we reconduct ourselves to work with the $N$-dimensional
sphere $S^N\,=\,\R^N\cup\{\infty\}$. \\
Obviously, $\R^N\!\setminus\!\Sigma$ and $S^N\!\setminus\!\Sigma$ have the same number of connected components.

By Alexander duality theorem (see theorem 3.44 of \cite{Hat}) with coefficients in $\mbb{Z}_2$, we have
$$
\wtilde{H}_k(S^N\!\setminus\!\Sigma;\Z_2)\,\simeq\,\wtilde{H}^{N-k-1}(\Sigma;\Z_2)\qquad\forall\,\,k\geq0\,.
$$
In particular, this is true for $k=0$:
$$
\wtilde{H}_0(S^N\!\setminus\!\Sigma;\Z_2)\,\simeq\,\wtilde{H}^{N-1}(\Sigma;\Z_2)\,.
$$

Now, as $\Sigma$ is compact, connected and without boundary, theorem 3.26 of \cite{Hat} applies, and gives us
$$
\wtilde{H}^{N-1}(\Sigma;\Z_2)\,\simeq\,\Z_2
$$
(independentely whether $\Sigma$ is orientable or not). In particular, also $\wtilde{H}_0(S^N\!\setminus\!\Sigma;\Z_2)$
is isomorphic to the same group, and this implies that the homology group (not reduced!)
\begin{equation} \label{iso:cohomol}
H_0(S^N\!\setminus\!\Sigma;\Z_2)\,\simeq\,\wtilde{H}_0(S^N\!\setminus\!\Sigma;\Z_2)\,\oplus\,\Z_2
\end{equation}
has rank equal to $2$. But the rank of $H_0(\mc{M};G)$ is always the number of the connected components of $\mc{M}$.
Hence, $S^N\!\setminus\!\Sigma$ has two connected components, $A$ and $B$.

Let us suppose that $\infty\in A$; then
$$
S^N\!\setminus\!\Sigma\;=\;A\,\cup\,B\qquad\Longrightarrow\qquad
\R^N\!\setminus\!\Sigma\;=\;(A\!\setminus\!\{\infty\})\,\cup\,B\,.
$$
Now, as $N\geq2$, $U:=A\!\setminus\!\{\infty\}$ is still connected.

Hence, $U$ and $B$ are the two connected components of $\R^N\!\setminus\!\Sigma$. \\
Moreover, it's easy to see (for instance, by stereographic projection with respect to the point $\infty$)
that $U$ is unbounded, while $B$ is bounded. \\
Finally, obviously $\d B\,\equiv\,\d U\,\equiv\,\Sigma$.
\end{proof}

As already pointed out, theorem \ref{th:piana} entails the following fundamental result. Even if it lies outside of the
topics of the present paper, we decided to quote it to give a more complete and detailed picture of the framework
we adopted in section \ref{s:conormal}.

Again, the proof is due to A. Lerario.

\begin{theo} \label{th:orient}
 Let $\Sigma\subset\R^N$ (for some $N\geq2$) be a compact, connected hypersurface without boundary.

Then $\Sigma$ is orientable.
\end{theo}

\begin{proof}
 The starting point is relation \eqref{iso:cohomol} in the previous proof. Actually, it holds true for any
submanifold $\mc{M}$ and any abelian group $G$:
\begin{equation} \label{iso:h_tilde}
H_0(\mc{M};G)\,\simeq\,\wtilde{H}_0(\mc{M};G)\,\oplus\,G\,.
\end{equation}
Moreover, it is always true that $H_0(\mc{M};G)$ is isomorphic to the direct product of $n$ copies of $G$,
where $n$ is the number of connected components of $\mc{M}$:
\begin{equation} \label{iso:conn_comp}
H_0(\mc{M};G)\,\simeq\,G^{\oplus n}
\end{equation}
(see \cite{Hat} for the proof of these facts).

In the previous proof, we established that the rank of $H_0(S^N\!\setminus\!\Sigma;\Z_2)$ is $2$. Then,
by \eqref{iso:conn_comp} we have that it is still $2$ if we consider the homology with coefficients in $\Z$:
$$
rk\bigl(H_0(S^N\!\setminus\!\Sigma;\Z)\bigr)\,=\,2\,.
$$
Therefore, keeping in mind \eqref{iso:h_tilde} and the Alexander duality theorem, we gather
$$
\wtilde{H}_0(S^N\!\setminus\!\Sigma;\Z)\,\simeq\,\Z\qquad\Longrightarrow\qquad
\wtilde{H}^{N-1}(\Sigma;\Z)\,\simeq\,\Z\,.
$$
Now, by theorem 3.26 of \cite{Hat}, this last condition is equivalent to the fact that $\Sigma$ is orientable.
\end{proof}


\end{document}